\documentclass[11 pt, leqno]{amsart}

\usepackage {amsfonts}
\usepackage{amsthm}
\usepackage{amssymb}
\usepackage{latexsym}
\usepackage{amsmath}
\usepackage{mathrsfs}
\usepackage{a4wide}
\usepackage{comment}
\usepackage{color}
\usepackage[normalem]{ulem}

\theoremstyle{plain}
\newtheorem{tw}{Theorem}[section]

\newtheorem {lem} [tw]{Lemma}
\newtheorem {prop}[tw] {Proposition}

\newtheorem{cor}[tw]{Corollary}

\theoremstyle{definition}
\newtheorem {deft}[tw] {Definition}
\newtheorem {rem} [tw]{Remark}
\newtheorem{example}[tw]{Example}

\newcommand{\bc} {\mathbb C}
\newcommand{\bn}{\mathbb N}
\newcommand{\br}{\mathbb R}

\newcommand{\BGW}{K\la d\ra}

\newcommand{\one}{{\bf 1}}

\newcommand {\Pol} {{\textup{Pol}}}

\newcommand {\id} {{\textrm{id}}}

\newcommand{\QG}{\mathbb{G}}
\newcommand{\QH}{\mathbb{H}}

\newcommand{\Com}{\Delta}
\newcommand{\Cou}{\epsilon}

\newcommand{\la}{\langle}
\newcommand{\ra}{\rangle}

\newcommand{\ot}{\otimes}

\newcommand{\wt}{\widetilde}

\numberwithin{equation}{section}

\begin{document}

\title[L\'evy-Khintchine decompositions for universal quantum groups]{L\'evy-Khintchine decompositions for generating functionals on algebras associated to universal compact quantum groups}

\begin{abstract}
We study the first and second cohomology groups of the $^*$-algebras of the universal 
unitary and orthogonal quantum groups $U_F^+$ and $O_F^+$. This provides 
valuable information for constructing and classifying L\'evy processes on these 
quantum groups, as pointed out by Sch\"urmann. In the case when all 
eigenvalues of $F^*F$ are distinct, we show that these $^*$-algebras have the 
properties (GC), (NC), and (LK) introduced by Sch\"urmann and studied recently 
by Franz, Gerhold and Thom. In the degenerate case $F=I_d$, we show that they do 
not have any of these properties. We also compute  the second cohomology group 
of $U_d^+$ with trivial coefficients -- $H^2(U_d^+,{}_\Cou\bc_\Cou)\cong \bc^{d^2-1}$ -- and 
construct an explicit basis for the corresponding second cohomology group for 
$O_d^+$ (whose dimension was known earlier thanks to the work of Collins, 
H\"artel and Thom). 
\end{abstract}


\author{Biswarup Das}
\address{Instytut Matematyczny, Uniwersytet Wroc\l awski, pl.Grunwaldzki 2/4, 50-384 Wroc\l aw, Poland}
\email{biswarup.das@math.uni.wroc.pl}

\author{Uwe Franz}
\address{Laboratoire de math\'ematiques de Besan\c{c}on,
Universit\'e de Bourgogne Franche-Comt\'e, 16, route de Gray, 25 030
Besan\c{c}on cedex, France}
\email{uwe.franz@univ-fcomte.fr}
\urladdr{http://lmb.univ-fcomte.fr/uwe-franz}

\author{Anna Kula}
\address{Instytut Matematyczny, Uniwersytet Wroc\l awski, pl.Grunwaldzki 2/4, 50-384 Wroc\l aw, Poland}
\email{Anna.Kula@math.uni.wroc.pl}

\author{Adam Skalski}
\address{Institute of Mathematics of the Polish Academy of Sciences,
ul.~\'Sniadeckich 8, 00--656 Warszawa, Poland
}
\email{a.skalski@impan.pl}

\maketitle

\section*{Introduction}

 The theory of L\'evy processes on involutive bialgebras subsumes both the 
theory of L\'evy processes with values in topological groups and the theory of 
factorizable representations of current groups and algebras. The first of these 
two theories describes stochastic processes with independent and stationary 
increments. Its paradigms are Brownian motion and the Poisson process and it 
played a central role in the invention of classical stochastic calculus, see 
\cite{app1,app2} and the references therein. The second was motivated by quantum 
field theory \cite{araki} and was instrumental in the development of quantum 
stochastic calculus, cf.\ \cite{acc,str}.

 Noncommutative L\'evy processes (initially called ``white noises'') were first 
introduced on $\mathbb{Z}_2$-graded involutive bialgebras \cite{asw} and later 
extended to arbitrary graded (or braided) involutive bialgebras 
\cite{schurmann93} and dual groups \cite{schurmann95}. It was soon realized that 
cohomology plays an important role in classifying and constructing L\'evy 
processes on a given $^*$-bialgebra, see \cite{schurmann90,schurmann93} and also 
\cite{FGT15,bfg} for recent progress in this direction.

 Based on Sch\"urmann's work \cite{schurmann93}, the standard approach to 
classifying L\'evy processes on a given involutive bialgebra or dual group 
consists in classifying first its $^*$-representations acting on a pre-Hilbert 
space and then determining their first cohomology group as well as the second 
cohomology group with trivial coefficients. For the Woronowicz quantum group $SU_q(2)$ 
this programme was carried out by Sch\"urmann and Skeide in the nineties 
\cite{schurmann+skeide98,skeidePhD}. Knowing the L\'evy processes on some quantum group gives geometric information on its structure and can be used to construct Dirac operators \cite{FabioUweAnna}.

If the C$^*$-algebra associated to a given compact quantum group is of type I, as it happens for $SU_q(2)$, then all $^*$-representations of the C$^*$-algebra are direct integrals of irreducible ones. In this case, if the irreducible $^*$-representations are known, it can be feasible to use the approach outlined above to classify all L\'evy processes on a given involutive bialgebra.

But there exist many interesting compact quantum groups whose C$^*$-algebras 
are not of type I and whose representation theory is too wild to be classified. This is the case of the universal unitary $U_F^+$ and the universal orthogonal $O_F^+$ quantum groups (see Examples \ref{ex_unitary} and \ref{ex_orth} for definitions). Then one can still gain a better understanding by studying the second cohomology group with trivial coefficients and the first cohomology group of some special $^*$-representations, as was done for quantum permutation groups in 
\cite{FKS}. In this paper we start from Sch\"urmann's work 
\cite{schurmann90,schurmann93} on the Brown-Glockner-von Waldenfels algebra \cite{brown,gvw}, i.e.\ the $^*$-algebra generated by the coefficients of a 
$d\times d$ unitary matrix with noncommuting entries, see Example \ref{ex_bgv}, and use his results to study 
the universal unitary and orthogonal quantum groups $U_F^+$ and $O_F^+$.

 In our study we consider mainly the following two extremal cases: the so-called generic 
case, that is the case where the eigenvalues of $F^*F$ are pairwise distinct, see Section 
\ref{sec-generic}, and the maximally degenerate case where $F=I_d$, see Section 
\ref{sec-degenerate}. In the generic case we show that all Gaussian cocycles  on $^*$-algebras of the compact quantum groups $U_F^+$ and 
$O_F^+$ admit a generating functional and can therefore be used to construct 
L\'evy processes on these quantum groups, see Theorems \ref{thm_au_gauss} and 
\ref{thm_aoq_decomp}. Using the terminology of \cite{FGT15}, see also Definition 
\ref{def-properties}, this means that these $^*$-algebras have the property 
(GC). This implies that all their generating functionals admit a 
L\'evy-Khintchine-type decomposition into a Gaussian part and a purely 
non-Gaussian part. Using similar methods, we show that the $^*$-algebra associated to the compact quantum group $SU_q(3)$ introduced by Woronowicz \cite{woronowicz88} does not have property (GC), see  Proposition \ref{prop_suqn_gc}. This allows us to conclude that the (GC) property need not pass neither from a quantum group to its quantum subgroup, nor from a quantum subgroup to the whole quantum group.

 In Section \ref{sec-degenerate} we consider the maximally degenerate case 
$F=I_d$ and classify the cocycles of the one-dimensional representations of the 
$^*$-algebras of the compact quantum groups $U_d^+:=U_{I_d}^+$ and 
$O_d^+:=O_{I_d}^+$. We also give necessary and sufficient conditions for these 
cocycles to admit a generating functional and therefore a L\'evy process. It 
follows from these results that  the $^*$-algebras of the compact quantum groups $U_d^+$ with 
$d\ge 2$ and $O_d^+$ with $d\ge 3$ have none of the properties (GC), (NC), or 
(LK), %defined in \cite{FGT15}, 
see Theorems \ref{thm-non-generic U_F} and 
\ref{thm-non-generic O_F}. This provides the first examples of  quantum groups 
which are non-cocommutative and do not have property (LK), and also has certain 
consequences for general non-generic $U_F^+$ and $O_F^+$. 

 The last section of this paper contains computations of the second cohomology group
with trivial coefficients of these $^*$-algebras.  Collins, H\"artel, and Thom 
\cite{CHT} (see also \cite{BichonCompositio}) have shown that the second 
cohomology group with trivial coefficients of 
the $^*$-algebra of $O_d^+$ has dimension $\frac{d(d-1)}{2}$. Here we show that 
 $H^2(U^+_d,{}_\Cou\bc_\Cou)$, the second cohomology group of  the $^*$-algebra of $U_d^+$ with trivial coefficients, has dimension $d^2-1$ (Theorem 
\ref{thm_main_cohomology_Ud}) and describe an explicit basis for it. For that, 
we define a 'defect map' which measures how far a cocycle is from being a 
coboundary, and show that this map induces an isomorphism between  
$H^2(U^+_d,{}_\Cou\bc_\Cou)$ and $sl(d,\bc)$, the space of complex $d\times d$ 
matrices with trace zero. A similar method establishes an isomorphism between 
$H^2(O_d^+,{}_\Cou\bc_\Cou)$ and $o(d,\bc)$, the space of complex 
antisymmetic $d \times d$ matrices, and also provides a basis of the former. 

\section{Notations and preliminaries} \label{NotPrel}

\subsection{Unitary quotient algebras} \label{examples}

The object of our study will be augmented algebras $(A,\Cou)$, i.e.\ unital $*$-algebras (over $\bc$)  equipped with a unital $*$-homomorphism $\Cou:A\to \bc$ (a character). We are going to consider only a special kind of augmented algebras. We will assume that $A$ admits a collection of elements $\{u_{jk}:j,k=1\ldots, d\}$ (where $d$ is some positive integer) generating $A$ as a $^*$-algebra, such that the matrix $u= (u_{jk})_{j,k=1}^d\in 
M_d(A)$ is unitary (that is, $uu^* = u^*u = I_{M_d(A)})$ and that the character satisfies the formula 
\begin{equation} \label{Def:counit}
\Cou(u_{jk}) = \delta_{jk}, \;\;j,k =1, \ldots, d .
\end{equation}
We shall call such pairs $(A,\Cou)$ for short \emph{unitary quotient algebras}, and the reason for that will become clear soon. If the set of generators is fixed, then we will denote a unitary quotient algebra by $(A,u)$, and will occasionally write $d=\dim u$. The character on $(A,u)$ is uniquely determined by the formula \eqref{Def:counit} (note that its existence is assumed in the definition of a unitary quotient algebra). 

The key example to keep in mind is the following.

\begin{example}{{\bf (Brown-Glockner-von Waldenfels algebra)}} \label{ex_bgv}
The \emph{Brown-Glockner-von Waldenfels 
algebra} (also called the \emph{noncommutative analogue of the algebra of coefficients of the unitary group}) $\BGW$, $d \in \bn$, is the universal unital $^*$-algebra with $d^2$ generators $x_{jk}$ ($j,k=1,2,\ldots,d$) and the relations
making the matrix $x:=(x_{jk})_{j,k=1}^d$ unitary. Then $(\BGW, x)$ is a unitary quotient algebra, and all other unitary quotient algebras are quotients of $(\BGW,x)$.
\end{example}

More examples of unitary quotients algebras come from compact matrix quantum groups in the sense of Woronowicz (\cite{wor1}). More precisely, every compact matrix quantum  group $\QG=(\mathsf{A},u)$ with 
a unital C$^*$-algebra $\mathsf{A}$ and a fundamental representation $u$ of $\QG$, gives rise to a structure of a unitary quotient algebra $(A,u)$ with $A$ 
being the unital $^*$-algebra $\Pol(\QG)$ spanned by coefficients of irreducible representations of $\QG$ and the same $u$. The character $\Cou$ is the counit. Note that the pair $(\Pol(\QG),\Cou)$ is not only an augmented algebra, but also has the structure of a Hopf $*$-algebra (i.e.\ it admits a comultiplication and an antipode). In the situations where this additional structure plays a role (e.g.\ Lemma \ref{lemma_eta_S}), it will be convenient to refer to unitary quotient algebras $(A,u)$ arising from compact matrix quantum groups as \emph{CMQG} (compact-matrix-quantum-group) \emph{algebras}. 

Contrary to the Brown-Glockner-von Waldenfels algebra, the canonical generators of CMQG algebras will always satisfy certain additional relations required by the fact that also $u^t$, the transpose of $u$, needs to be invertible. Here, we use the following notation: for a matrix $a = (a_{jk})_{j,k=1}^d \in M_d(A)$ of elements of a $^*$-algebra, its transpose will be  $a^t:= (a_{kj})_{j,k=1}^d$ and its entry-wise conjugate $\bar{a}:= (a_{jk}^*)_{j,k=1}^d$.

\medskip
A crucial notion in our approach is that of a quotient. Given two unitary quotient algebra, $(A,u)$ and $(B,v)$, we will say that $(B,v)$ is a \emph{quotient} of $(A,u)$ if $\dim u = \dim v$ and there is a unital $^*$-homomorphism $q: A \to B$ such that
\[ q(u_{jk}) = v_{jk}, \;\;\; j,k=1,\ldots, \dim v.\]
Note that then $q$ is automatically surjective, and preserves the character (we have \linebreak $(\Cou_B\circ q)(u_{jk}) = \Cou_B(v_{jk})=\delta_{jk} = \Cou_A(u_{jk})$). As mentioned in Example \ref{ex_bgv}, if $(A,u)$ is a unitary quotient algebra with $d={\rm dim }\ u$, then it is a quotient of $\BGW$ (noncommutative analogue of the algebra of coefficients of the unitary group), which explains the terminology. Let us also remark that if $A=\Pol(\QG)$, and $A=\Pol(\QH)$, for $\QG, \QH$ compact matrix quantum groups and with $u$, $v$ being respective fundamental representations of $\QG$ and $\QH$, the above implies in particular that $\QH$ is a quantum subgroup of $\QG$ (in the sense introduced in \cite{Podles} and later studied for example in \cite{DKSS}).

The examples of unitary quotient algebras which are of importance for our considerations are the following.

\begin{example}{{\bf (Universal unitary quantum group)}}\label{ex_unitary}
Consider a matrix $F\in GL_d(\bc)$ ($d \in \bn$) and the universal unital $^*$-algebra generated by $d^2$ elements $u_{jk}$ ($j,k=1,2,...d$) such that the matrix
$u:=(u_{jk})_{j,k=1}^d$ is unitary and its conjugate, $\bar{u}$, is similar to a unitary via $F$. More specifically, 
\begin{equation} \label{eq_rel1_auf}
{\mbox{(R1)}} \ \ uu^*=I=u^*u ; \qquad
{\mbox{(R2)}} \ \ 
F\bar{u}F^{-1}(F\bar{u}F^{-1})^*=I=(F\bar{u}F^{-1})^*F\bar{u}F^{-1}.
\end{equation}
The algebra introduced above, denoted by $\Pol(U_F^+)$,  is a CMQG algebra of the compact quantum group  $U^+_F$,  called the \emph{universal unitary quantum group} (in the sense of Banica, cf.\ \cite{banica97}). Hence the pair $(\Pol(U_F^+), u)$ is a unitary quotient algebra. In the   particular case when $F=I_d$ we write $U^+_d:=U^+_{I_d}$. The compact quantum group $U^+_d$ is called the \emph{free unitary quantum group}. 

 The universality of the family $\{U^+_F;K\in GL_d(\bc)\}$ is understood as follows (cf.\ \cite{banica97}). Given a compact quantum matrix group $\QG$ with the fundamental representation $w:=(w_{jk})_{j,k=1}^d\in M_d(\Pol(\QG))$, we know from \cite{wor1} that the matrix $\bar{w}$ is equivalent to a
unitary representation, hence there exists an invertible complex matrix $F\in M_d(\bc)$
such that $F\bar{w}F^{-1}$ is unitary. Then $(\Pol(\QG),w)$ is a quotient of $(\Pol(U^+_F),u)$.

Let us observe that the relation (R2) is equivalent to the equality 
\begin{equation} \label{eq_rel3_auf}\tag{R3}
u^tQ\overline{u}Q^{-1}=I=Q\overline{u}Q^{-1}u^t.
\end{equation}
with the positive matrix $Q=F^*F$. Thus $\Pol(U_F^+)$
is
isomorphic to the CMQG algebra associated with the universal unitary quantum group in the sense of Van Daele and Wang, cf.\
\cite{wang+vandaele96}, defined as the universal unital $^*$-algebra $A_u^{W}(Q)$ generated by the coefficients of $u:=(u_{jk})_{j,k=1}^d$ satisfying (R1) and
(R3).
In other words, $\Pol(U_F^+)$ is the quotient of $\BGW$ by the ideal generated by the relations \eqref{eq_rel3_auf}, i.e. $u^tQ\overline{u}Q^{-1}=I=Q\overline{u}Q^{-1}u^t$ with $Q=F^*F$.
\end{example}

\begin{rem}\label{rem_abstract_isomorphism}
It follows from Proposition 6.4.7 in \cite{timmermann} that up to an isomorphism of CMQG algebras without loss
of generality we can assume that the matrix $Q$ is diagonal. In that case, the relation \eqref{eq_rel3_auf} reads 
\begin{equation}\label{eqn_rel_Au}
\sum_{p=1}^{d}\frac{Q_{pp}}{Q_{kk}} u_{p j}u^*_{pk}=\delta_{jk}1, \qquad \;\;
\sum_{p=1}^{d}\frac{Q_{jj}}{Q_{pp}}u^*_{jp}u_{kp}=\delta_{jk}1, 
\;\;\; j,k=1,\ldots, d.
\end{equation}
\end{rem}

\begin{example}(Quantum group $SU_q(d)$) \label{ex_suqd}
Let $d \in\bn$ and $q\in (0,1)$. The algebra $\Pol(SU_q(d))$ is defined  as the universal 
unital $^*$-algebra generated by the matrix coefficients of the matrix $u = 
(u_{jk})_{j,k=1}^d$ satisfying the unitarity condition (R1) and the quantum 
determinant condition 
\begin{equation} \label{eq_td}
\sum_{\sigma\in S_d} (-q)^{i(\sigma)}
u_{\sigma(1),\tau(1)}u_{\sigma(2),\tau(2)}\ldots u_{\sigma(d),\tau(d)} =
(-q)^{i(\tau)}\cdot 1, \;\;\; \tau \in S_d,
\end{equation}
where $S_d$ denotes the permutation group on $d$ letters and $i(\tau)$ denotes 
the number of inversions of the permutation $\tau$ (see \cite{woronowicz88}).

It is known, see \cite[Lemma 4.7]{vandaele93}, that the adjoint $\bar{u}$ of the representation $u$ of $SU_q(d)$, is similar to a unitary via the matrix $F= (q^{j-d}\delta_{jk})_{j,k=1}^d$. Hence the $^*$-algebra $\Pol(SU_q(d))$ is the quotient of $\Pol(U_F^+)$ by the ideal generated by the twisted determinant condition \eqref{eq_td}.
\end{example}

\begin{example}{{\bf (Universal orthogonal quantum group)}}\label{ex_orth}
Let $d \in \bn$ and let $F\in GL_d(\bc)$ satisfy the condition $F\bar{F}= \pm I$. The universal $^*$-algebra generated by $d^2$ elements $v_{jk}$, $ j,k =1, \ldots d$ subject to the condition that the matrix $v=(v_{jk})$ is unitary and $v=F\bar{v}F^{-1}$, is a CMQG algebra associated to the \emph{universal orthogonal
compact quantum group}, denoted by $O^+_F$ (\cite{banica96}). Therefore we denote it by $\Pol(O_F^+)$. Again we write $O^+_d:=O^+_{I_d}$ and call $O^+_d$ the \emph{free orthogonal quantum group}.

Note that $\Pol(O_F^+)$ is the quotient of $\BGW$ and of $\Pol(U_F^+)$ by the ideal generated by the relations $u=F\bar{u}F^{-1}$, or equivalently, $uF=F\bar{u}$. In particular, $\Pol(O_d^+)$ is the quotient of $\BGW$ by the ideal generated by the relations $u_{jk}=u_{jk}^*$, $j,k=1,\ldots,d$. 
\end{example}

\begin{rem} \label{rem_matrices_F}
Different matrices $F$ can lead to isomorphic quantum groups: the quantum
groups
$O^+_F$ and $O_{F'}^+$ are isomorphic (so in particular $\Pol(O_F^+)$ and $\Pol(O_{F'}^+)$ are isomorphic as $^*$-algebras with characters) if and only if $F'=UFU^t$ for a
unitary
matrix $U$ by \cite[Proposition 6.4.7]{timmermann} (see also \cite{wang+vandaele96}). The matrices $F$ representing
each class of isomorphic quantum groups can be classified as follows (see
\cite[Remark 1.5.2]{derijdt}):
\begin{enumerate}
 \item either there exists a diagonal matrix $D={\rm diag} (\lambda_1, \ldots,
\lambda_p)$ with $0<\lambda_1 \leq \ldots \leq \lambda_p<1$ and then  $$
F=\left( \begin{array}{ccccccc}
 0 & D & 0\\ D^{-1} & 0 & 0 \\ 0 & 0 & I_{d-2p}
 \end{array}\right),$$
 \item or there exists a diagonal matrix $D={\rm diag} (\lambda_1, \ldots,
\lambda_\frac{d}{2})$ with $0<\lambda_1 \leq \ldots \leq
\lambda_\frac{d}{2}\leq 1$ and then  $$
F=\left( \begin{array}{ccccccc}
 0 & D \\ -D^{-1} & 0
 \end{array}\right).$$
\end{enumerate}
Let us finally note that, in the definition of $O_F^+$, we could have allowed also matrices $F$ such that $F\bar{F} =c I$ for any non-zero $c \in \br$, but this does not increase the class of quantum groups in question.
\end{rem}

\begin{example}{{\bf (Cocommutative quantum groups)}} \label{ex_cocomm}
Let $\Gamma$ be a finitely generated group, with the fixed generating 
set $\gamma_1,\ldots, \gamma_d$. Then the complex group algebra $\bc[\Gamma]$ with the matrix $u\in M_d(\bc[\Gamma])$ given by the formula 
$u=(\delta_{jk} \gamma_j)_{j,k=1}^d$ is a unitary quotient algebra.  In fact, $(\bc[\Gamma],u)$ is a \emph{cocommutative} (sometimes also called \emph{abelian}) CMQG algebra. One shows that all cocommutative CMQG algebras arise in this way, and we view the corresponding compact quantum group as the dual of $\Gamma$, writing $\bc[\Gamma] = \Pol(\hat{\Gamma})$.
\end{example}

\subsection{Properties (LK), (NC), (GC) -- definitions and known results} \label{ssec_schurman}

Let $(A, \Cou)$ be a unital $^*$-algebra with a character (for simplicity we assume from the beginning that $A$ is a unitary quotient algebra).
We say that a linear map $\psi:A
\to \bc$ is a \emph{generating functional} if it is:
\begin{itemize}
\item normalized, i.e.\ $\psi(1)=0$;
\item hermitian, i.e.\ $\psi(a^*)=\overline{\psi(a)}$ for all $a\in A$;
\item conditionally positive, i.e. $\psi(a^*a)\geq 0$ for $a\in \ker \Cou$.
\end{itemize}
The name comes from the correspondence between such functionals and weakly
continuous convolution semigroups of states on $A$, for which the generating functionals
are infinitesimal generators. Such  semigroups are in turn in one-to-one
correspondence with L\'evy processes on $A$, cf.\ \cite[Corollary 1.9.7]{schurmann93}, \cite[Proposition 4.3.2]{franz+schott}. We refer to \cite{schurmann93} and \cite{schurmann90} for more information on the related concepts.

\begin{deft}
A \emph{Sch\"urmann triple} on an augmented algebra $(A,\Cou)$ is a triple $(\rho, \eta,\psi)$, consisting of a
unital $^*$-representation $\rho:A \to L_{\rm ad}(D)$ by adjointable linear operators on a pre-Hilbert space $D$, a linear map  $\eta: A \to D$ such that
\begin{equation} \label{eq_def_eta}
\eta(ab)=\rho(a)\eta(b)+\eta(a)\Cou(b), \quad a,b\in A,
\end{equation}
and a linear functional $\psi:A\to \bc$ such that
\begin{equation} \label{eq_def_psi}
\psi(ab)=\psi(a)\Cou(b)+\Cou(a)\psi(b)+ \langle \eta(a^*),\eta(b)\rangle, \quad
a,b\in A,
\end{equation}
\end{deft}

Condition \eqref{eq_def_eta} means that $\eta$ is a $1$-$\rho$-$\Cou$-cocycle, and Condition \eqref{eq_def_psi} is equivalent to the fact that the bilinear map $A\times A\ni(a,b)\mapsto \langle\eta(a^*),\eta(b)\rangle\in\bc$ is the $2$-$\Cou$-$\Cou$-coboundary of $\psi$. See the beginning of Section \ref{sec_cohom} for a systematic introduction of the terminology of Hochschild cohomology.

Sch\"urmann showed via a GNS-type construction that for every generating functional such a triple exists, cf.\ \cite[Theorem 2.3.4]{schurmann93}, \cite[Section 4.4]{franz+schott}.

The unitarity relations imply that the generators $\{u_{jk};j,k=1,\ldots,d\}$ of a unitary quotient algebra $A$ are always represented by bounded operators, and therefore we can extend $\rho(a)$ for all $a\in A$ to a bounded operator on the completion $H=\overline{D}$. For this reason we can view the maps $\rho$ and $\eta$ of a Sch\"urmann triple on a unitary quotient algebra as taking values in $B(H)$ and $H$, resp., with $H$ a Hilbert space.

The triple is unique, up to a suitable notion of a unitary isomorphism, provided $\eta(A)$ is dense in $H$. It is called \emph{the Sch\"urmann triple} on $A$ associated to $\psi$. Conversely, given any Sch\"urmann triple on $A$ the `hermitianisation' of its third ingredient (i.e.\ the functional obtained by passing from $\psi$ to $\frac{1}{2} (\psi + \bar{\psi})$, where $\bar{\psi}(a):= \overline{\psi(a^*)}$ for all $ a \in A$) is a generating functional. Thus we have a natural one-to-one
correspondence between Sch\"urmann triples on $A$ with a hermitian functional and generating functionals on $A$. 

In what follows whenever we speak about a representation of a unital $^*$-algebra on a Hilbert space we mean a unital $^*$-representation. 

We say that a pair $(\rho,\eta)$ of a representation $\rho$ and a $\rho$-$\Cou$-cocycle $\eta$ \emph{admits a generating functional} if there exists a linear functional $\psi$ such that $(\rho,\eta,\psi)$ is a Sch\"urmann triple.
In general, a pair $(\rho,\eta)$, consisting of a representation and a
$\rho$-$\Cou$-cocycle, need not admit a generating functional. This is basically
due to the coboundary condition \eqref{eq_def_psi} which implies that the values of $\psi$ on $K_2={\rm Span}\, \{ab;a,b\in \ker \Cou\}$ have to satisfy $\psi(ab)=\langle\eta(a^*),\eta(b)\rangle$ for $a,b\in \ker \Cou$. It is easy to find an augmented algebra $(A,\Cou)$ with a cocycle $\eta$ and elements $a,b,a',b'\in \ker \Cou$ such that $ab=a'b'$, but $\langle\eta(a^*),\eta(b)\rangle \not=\langle\eta((a')^*),\eta(b')\rangle$, cf.\ \cite[Example 2.1]{skeide99b} or the proof of Proposition \ref{prop_suqn_gc}.

A generating functional $\psi$ is called \emph{Gaussian} if it vanishes on
triple products of elements from the kernel of the counit:
$ \psi(abc)=0$ for $a,b,c\in \ker \Cou$. This is the case if and only if the
associated representation $\rho$ in the Sch\"urmann triple $(\rho, \eta,
\psi)$ is of the form $\rho= \Cou(\cdot) \id$ or equivalently
if the cocycle $\eta$ is \emph{Gaussian}, i.e.\
$\eta(ab)=\Cou(a)\eta(b)+\eta(a)\Cou(b)$, $a, b \in A$.

Given a representation $\rho$ of $A$ on $H$ we can always extract the maximal
Gaussian subspace of the representation space $H$, that is the space
$H_G:=\{\eta \in H: \rho(a) \eta = \Cou(a) \eta, a \in A\}$. If $H_G=\{0\}$, then we say that the
representation (and the cocycle, and the associated generating functional) is
\emph{purely non-Gaussian}.

Let $\psi: A \to \bc$ be a generating functional and let $(\rho, \eta, \psi)$ be the
associated Sch\"urmann triple with the representation space $H$. Let
$P_G$ be the projection from $H$ onto the maximal Gaussian subspace $H_G$
and set
$$\rho_G=\rho|_{H_G}, \, \rho_N = \rho|_{H_G^\perp}, \quad
\mbox{and} \quad \eta_G=P_G\eta, \,\eta_N=(I-P_G)\eta.$$
Then $\eta_G$ is a Gaussian cocycle and $\eta_N$ is a purely non-Gaussian
cocycle. If there exist generating functionals $\psi_G$ and $\psi_N$ such that
the $(\rho_G,\eta_G,\psi_G)$ and $(\rho_N,\eta_N,\psi_N)$ are Sch\"urmann
triples, then we say that $\psi$ admits a \emph{L\'evy-Khintchine
decomposition}.

If for a given $\psi$ one of the pairs
$(\rho_G,\eta_G)$ or $(\rho_N,\eta_N)$ admits a generating functional, say $\psi_G$, then it follows that $\psi$ admits a L\'evy-Khintchine decomposition, since $\psi-\psi_G$ will be a generating functional for $(\rho_N,\eta_N)$. In fact it is easy to see that $\psi$ admits a L\'evy-Khintchine decomposition if and only if it can be written as a sum of two generating functionals, one of which is Gaussian, and the other purely non-Gaussian. This motivated our terminology which comes from the analogy to the L\'evy-Khintchine formula in the classical case, where a the characteristic function of a L\'evy process is commonly written as the exponential of a sum of a Gaussian term and a purely non-Gaussian term.

Adopting the notation from \cite{FGT15}, we shall study the following three properties in this paper.

\begin{deft}\label{def-properties}
We shall say that the augmented algebra $(A,\Cou)$
\begin{itemize}
	\item 
has the property (LK) if any
generating functional on $(A,\Cou)$ admits a L\'evy-Khintchine decomposition;
\item has the property (GC) if any Gaussian
cocycle $\eta:A\to H$ can be completed to a Sch\"urmann triple $(\Cou(\cdot)\id, \eta,
\psi)$;
\item has the property (NC) if any pair $(\rho,
\eta)$ consisting of a representation $\rho:A \to B(H)$ and a
$\rho$-$\Cou$-cocycle $\eta:A\to H$ with $H_G=\{0\}$ can be completed to a
Sch\"urmann triple $(\rho, \eta, \psi)$. 
\end{itemize}
\end{deft}
Property (LK) is equivalent to Sch\"urmann's property (C), property (GC) was called (C') by Sch\"urmann, and property (NC) was not explicitly named by Sch\"urmann, cf.\ \cite[5.1 Maximal quadratic components]{schurmann93}.

If $A= \Pol(\QG)$ is a CMQG algebra, we will also say simply that the compact quantum group $\QG$ has the property (LK), (GC), or (NC).

As observed by Sch\"urmann, either (GC) or (NC) implies (LK). However, no two of the
conditions are equivalent, as shown by the examples of group rings provided in
\cite{FGT15}. There exists a cocommutative CMQG-algebra (i.e.\ a CMQG-algebra coming from a group ring of a finitely generated group, see Example \ref{ex_cocomm}) without the property (LK). On the other hand, any commutative CMQG-algebra (in other words, any classical compact matrix group) has (LK), see \cite[Theorem 3.12 (at the top of p.\ 360)]{schurmann90}.

For our considerations the key fact will be the following theorem of Sch\"urmann, implying in particular that the algebra $\BGW$ has all the properties listed in Definition \ref{def-properties}.

\begin{tw}{\cite[Theorem 3.12(i)]{schurmann90}} \label{thm_schurmann}
Let $d \in \bn$ and $\rho:\BGW \to B(H)$ be a representation of the Brown-Glockner-von Waldenfels algebra on a Hilbert space $H$. Any $\rho$-$\Cou$-cocycle $\eta$ on $\BGW$ admits a generating functional $\psi:\BGW \to \bc$ which is a coboundary for $\eta$. A possible choice of the functional $\psi$ is  determined
by the following formula describing its value on the canonical generators of $\BGW$ ($j,k=1,\ldots,d$):
\begin{equation} \label{eq_gen_fun_bgw}
\psi(x_{jk})
= -\frac{1}{2}\sum_{p=1}^{d}\la\eta(x^*_{jp}),\eta(x^*_{kp}
)\ra.
\end{equation}
\end{tw}

The values of $\psi$ on general elements of $A$ can then be deduced from the normalization $\psi(1)=0$, hermitianity, and the coboundary condition \eqref{eq_def_psi}.

In general we have the freedom to modify the generating functional by any selfadjoint matrix $H\in M_d(\bc)$, by setting
\[
\psi_H(x_{jk})
= -\frac{1}{2}\sum_{p=1}^{d}\la\eta(x^*_{jp}),\eta(x^*_{kp}
)\ra+i H_{jk}
\]
for $j,k=1,\ldots,d$.
%%%%%%%%%%%%%%%%%%%%%%%%%%%%%%%%%%%%%%%%%%%
\subsection{Sch\"urmann triples on unitary quotient algebras}

In this paper, an important strategy for associating a generating functional to a given cocycle on a unitary quotient algebra $(A,u)$  relies on the fact that $A$ is a quotient of $\BGW$, arising via imposing on the generators $u_{jk}$ some additional relations. We will explain it now.

Let $(A,u)$ be a unitary quotient algebra. 
Note first that since $A$ arises as the universal $^*$-algebra generated by the coefficients $(u_{jk}:j,k=1,\ldots,d)$ satisfying certain relations, then any representation $\rho$ of $B$, any $\rho$-$\Cou$-cocycle $\eta$ and a generating functional $\psi$ for $\eta$ (if it exists) are all determined by the values they take on the canonical generators and their adjoints (the values on $1$ are fixed: $\rho(1)=I$, $\eta(1)=0$ and $\psi(1)=0$). 

Observe that a given cocycle $\eta$ on $A$ lifts to a cocycle $\eta'$ on $\BGW$ simply by composition with the quotient map. We can therefore use Theorem \ref{thm_schurmann} and Sch\"urmann's formula \eqref{eq_gen_fun_bgw} to define a generating functional $\psi'$ on $\BGW$ associated to the cocycle $\eta'$. This generating functional will then descend to $A$ if and only if it preserves the extra relations in $A$, i.e., it vanishes on the kernel of the canonical quotient map from $\BGW$ to $A$. To show the existence of a generating functional  $\psi$ on $A$ it suffices to show that the values on the generators, proposed by \eqref{eq_gen_fun_bgw}, are compatible with the relations determining $A$. This result was proven in \cite{FKS} -- we recall it here as we will use it many times in this paper. 

\begin{lem}\cite[Lemma 5.8]{FKS}\label{lem_5.8} Let $B$ be a $^*$-algebra generated by a collection of elements, $a_1, \ldots,a_n$, let $\Cou$ be a character on $B$, and let $(\rho, \eta, \psi)$ be a Sch\"urmann triple on $B$. Let $A$ be the quotient of $B$ by the two-sided ideal generated by the selfadjoint relations
$r_1(a_1, \ldots,a_n)=0$, $\ldots$, $r_k(a_1, \ldots,a_n) = 0$.

If the maps $\Cou$, $\rho$, $\eta$, and $\psi$ vanish on $r_1, \ldots,r_k$, then $(\rho, \eta, \psi)$ is a Sch\"urmann triple on $A$.
\end{lem}

We finish this section by providing a general result concerning an automatic property of Gaussian cocycles on CMQG algebras. 

\begin{lem} \label{lemma_eta_S}
Suppose that $(A,u)$ is a CMQG algebra, $H$ is a Hilbert space and $\eta:A \to H$ is a Gaussian cocycle. Then 
\[ \eta(u_{jk}^*)= - \eta(u_{kj}), \;\;\; j,k=1,\ldots,d.\]
\end{lem}
\begin{proof}	
Note that the CMQG algebra $(A,u)$ has a Hopf $*$-algebra structure with comultiplication $\Delta$ 
%which acts on the generators as $\Delta(u_{jk})=\sum_{p} u_{jp}\otimes u_{pk}$ 
and an antipode $S$ which acts on the generators as $S(u_{jk})=u_{kj}^*$. 

Let $a\in A$, and use the Sweedler notation $\Com(a) = a_{(1)} \ot a_{(2)} \in A \ot A$.
 Applying $\eta$ to the formula $\Cou(a)1 = a_{(1)}S(a_{(2)})$ yields
\[
0=\eta(\Cou(a)1)=\eta(a_{(1)})\epsilon(a_{(2)})+\eta(Sa_{(1)})\epsilon(a_{(2)})= \eta(a)+\eta(Sa).
\]
This shows that $\eta = - \eta \circ S$, from which the statement follows.
\end{proof}

%%%%%%%%%%%%%%%%%%%%%%%%%%%%%%%%%%%%%%%%%%%%%%%%%%%%%%%%%%%%%%%%%%%%%%%
\section{Generic case for $U^+_F$ and $O^+_F$}
\label{sec-generic}

Throughout this section we consider the quantum groups $U^+_F$ and $O^+_F$ with 
a matrix $F\in GL_d(\bc)$ such that 
$Q=F^*F\in GL_d(\bc)$ has pairwise 
distinct eigenvalues. We will call such $F$ \emph{generic}.  Note that genericity is preserved under the transformations $F \mapsto w^* Fw$ and $F \mapsto w Fw^{t}$, if $w \in U_d(\bc)$.

\begin{tw} \label{thm_au_gauss}
Let $d \in \bn$ and let  $F\in GL_d(\bc)$ be a generic matrix. Then $U_d^+(F)$ has 
the property {\rm (GC)}, hence also the property {\rm (LK)}. 
\end{tw}

\begin{proof}
According to Remark \ref{rem_abstract_isomorphism}, we may and do assume that the matrix $Q=F^*F$ is diagonal and use the presentation of $\Pol(U_F^+)$ as $A:=A_u^W(Q)$. Suppose that $H$ is a Hilbert space and $\eta: \Pol(U_F^+) \to H$ is a Gaussian cocycle. 
Applying $\eta$ to both sides of 
the first equality in \eqref{eqn_rel_Au} yields ($j,k=1,\ldots, d$)
\[
0 = \frac{1}{Q_{kk}}\sum_{p=1}^{d}Q_{pp}[\eta(u_{pj})\delta_{pk}+\eta(u_{pk}
^*)\delta_{pj}]=
Q_{kk}\eta(u_{kj})+Q_{jj}\eta(u^*_{jk}). 
\]
Using Lemma \ref{lemma_eta_S} we obtain
\[
(Q_{kk}-Q_{jj})\eta(u_{jk})=0. 
\]
By genericity assumption $Q_{jj}\neq Q_{kk}$ for $j\neq k$. Thus  we must have 
\begin{equation}\eta(u_{jk})=0,\;\;\;\;\;\;\; j, k =1, \ldots, d,\quad j \neq k. 
\end{equation}

Let $q: \BGW\longrightarrow \Pol(U_F^+)$ be the canonical quotient map, that is the surjective $^*$-homo--morphism 
given by $q(x_{jk})=u_{jk}$, $j,k=1, \ldots, d$, where  
$\{x_{jk}:j,k=\ldots,d\}$ are generators of $\BGW$. The kernel of $q$ is the two-sided ideal generated by the relations
\begin{align}
r^{(1)}_{jk} :=\frac{1}{Q_{kk}}\sum_{p=1}^{d}Q_{pp}x_{pj}x^*_{pk}-\delta_{jk}1=0 
\quad \mbox{and}\quad 
 r^{(2)}_{jk} := Q_{jj}\sum_{p=1}^{d}\frac{1}{Q_{pp}}x^*_{jp}x_{kp}-\delta_{jk}1)=0,
\end{align}
$j,k=1,\ldots,d$. 

Define $\eta^\prime:=\eta\circ q: \BGW\longrightarrow H$. It is easy to check that 
$\eta^\prime$ is a Gaussian cocycle on $\BGW$. By Theorem 
\ref{thm_schurmann}, we can associate with $\eta'$ a generating functional
$\psi^\prime: \BGW\longrightarrow\bc$, determined by formula \eqref{eq_gen_fun_bgw}. We would like to define a functional $\psi:A\longrightarrow\bc$ such that
$\psi(q(a)):=\psi^\prime(a)$. According to Lemma \ref{lem_5.8}, to prove that $\psi$ is well-defined, it is enough to check that $\psi^\prime(r^{(m)}_{jk})=0$  for $j,k=1,\ldots,d$, $m=1,2$. We only show the case $m=1$ as the proof of the second case is very similar. 

By the first part of the proof $\eta$ (so also $\eta'$) vanishes on the off-diagonal generators, therefore the definition \eqref{eq_gen_fun_bgw} of $\psi'$ finally gets the form 
\begin{equation}\label{eq_psi_on_Kd}
\psi^\prime(x_{jk})=0\quad (j\neq k) 
\quad \mbox{and}\quad 
\psi^\prime(x_{kk})=-\frac{1}{2}\|\eta^\prime(x_{kk})\|^2. 
\end{equation}

For $j,k=1,\ldots, d$ and $j\neq k$ we have
\begin{align}\label{eqn1}
\psi^\prime(\sum_{p=1}^{d}Q_{pp}x_{pj}x^*_{pk})&=\sum_{p=1}^{d}Q_{pp}
\psi^\prime(x_{pj}x^*_{pk})\\
\nonumber
&=\sum_{p=1}^{d}Q_{pp}[\psi^\prime(x_{pj})\delta_{pk}+\delta_{pj}
\psi^\prime(x^*_{pk})+\la\eta^\prime(x^*_{pj}),\eta^\prime(x^*_{pk})\ra]\\
\nonumber
&=Q_{kk}\psi^\prime(x_{kj})+Q_{jj}\psi^\prime(x^*_{jk})
+\sum_{p=1}^{d}Q_{pp} \la\eta^\prime(x^*_{pj}),\eta^\prime(x^*_{pk})\ra.
\end{align}
The first two summands vanish due to the definition of $\psi'$, see \eqref{eq_psi_on_Kd}. Moreover by Lemma \ref{lemma_eta_S} $\eta^\prime(x^*_{jk})=\eta(u^*_{jk})=-\eta(u_{kj})=0$ for $j\neq k$, and hence the last summand must be proportional to $	\delta_{p,j}\delta_{p,k}$. Therefore, \eqref{eqn1} is zero for $j\neq k$. 
	
On the other hand if $j=k$, then again by Lemma \ref{lemma_eta_S} we get
\begin{align*}
\psi^\prime(\sum_{p=1}^{d}Q_{pp}x_{pj}x^*_{pj})&= Q_{jj}[\psi^\prime(x_{jj})+\psi^\prime(x^*_{jj})+\|\eta^\prime(x_{jj})\|^2]. 
\end{align*}

Using the fact that $\psi^\prime(x_{kk})=\psi^\prime(x_{kk}^*)=-\frac{1}{2}\|\eta^\prime(x_{kk})\|^2$ 
for all $k=1,2,...,d$, we see that the above expression also vanishes. Thus the 
prescription $\psi(q(a)):=\psi^\prime(a)$ for all $a\in A$ is well-defined. 
Now a routine check shows that  the coboundary of this 
functional is the bilinear form $(a,b)\mapsto\langle\eta(a^*),\eta(b)\rangle$.
\end{proof}

We shall see later (see 
Theorem \ref{thm-non-generic U_F}) that genericity of $F$ is not only 
sufficient, but also a necessary condition for the 
quantum 
group $U^+_F$ to have property (GC).

The analogue of Theorem \ref{thm_au_gauss} holds also in the orthogonal case, as 
we show in the next theorem.

\begin{tw} \label{thm_aoq_decomp}
Let $d \in \bn$ and let  $F\in GL(d)$ be a generic matrix such that $F\bar{F}= 
\pm I$. Then the quantum group $O_F^+$ has the 
property (GC), hence also the property (LK).
\end{tw}

\begin{proof}
Let $q: \Pol(U_F^+) \to \Pol(O_F^+)$ be the canonical  quotient map, i.e.\ the surjective $^*$-homomorphism  such that $q(u_{jk})=v_{jk}$, $j,k=1\ldots d$. The kernel of $q$ is generated by the 
relation $uF=F\bar{u}$, that is by the family of relations
$$\sum_{p=1}^d u_{jp} F_{p k} =  \sum_{p=1}^d F_{j p}u_{p k}^*, \;\;\;\; 
j,k=1\ldots d.
$$

Note first that, without loss of generality, we can assume that $F$ is of the form described in Remark \ref{rem_matrices_F}. 
 Then for each $j=1,\ldots, d$ there exists exactly one $\hat{j}\in \{1,\ldots, 
d\}$ such that
$F_{j\hat{j}}\neq 0$, and similarly, for each $k=1,\ldots, d$ there 
exists exactly one $\check{k}\in \{1,\ldots, d\}$ such that $F_{\check{k}k}\neq 
0$. The mapings $j\mapsto 
\hat{j}$ and $k\mapsto \check{k}$ are  inverses of one another. All this implies that the kernel of the quotient map $q$ is generated by the relation
\begin{equation} \label{eq_F-uniqueness}
u_{j\check{k}} F_{\check{k} k} =  F_{j\hat{j}}u_{\hat{j} k}^*, \;\;\;\; 
j,k=1,\ldots, d.
\end{equation}

Let us fix a Hilbert space $H$ and a Gaussian cocycle $\eta: \Pol(O_F^+)\longrightarrow H$. We will show that it admits a generating functional.
Note first that \eqref{eq_F-uniqueness} together with Lemma \ref{lemma_eta_S} imply that, for $j,k=1,\ldots, d$,
\begin{eqnarray*}
\eta(u_{j\check{k}}) F_{\check{k}k}=  - F_{j \hat{j}}\eta(u_{k\hat{j}} ) .
\end{eqnarray*} 
In particular $	\eta(u_{\check{k}\check{k}}) F_{\check{k}k} = - 
F_{\check{k}k}\eta(u_{kk})$, so that
\begin{equation} \label{eq_eta_gaussian_aof}
\eta(u_{kk}) = - 
\eta(u_{\check{k}\check{k}}).
\end{equation}

Define now $\eta_u:=\eta\circ q: U_F^+\longrightarrow H$. This is a 
Gaussian cocycle on $U_F^+$, so by Theorem \ref{thm_au_gauss} and its proof, it 
admits a generating  functional $\psi_u $ defined on the generators by the 
formulas ($j,k=1,\ldots,d$)
\[
\psi_u (u_{jk})=\begin{cases}
0 & \textup{for}\; j\neq k,  \\ -\frac{1}{2}\|\eta_u(u_{kk})\|^2 & 
\textup{for}\; j=k.
\end{cases}
\]

As before, we would like to define $\psi$ on $O_F^+$ by $\psi 
(q(a))=\psi_u(a), a \in O_F^+$. Such a functional -- if well-defined -- will 
indeed be a generating functional that we seek. 
It thus remains  to check that $\psi_u$ vanishes on $\ker q$, or by Lemma \ref{lem_5.8} 
that we have
\begin{equation} \label{eq_psi_uF-Fu}
\psi_u (u_{j\check{k}} F_{\check{k}k})=  \psi_u (F_{j \hat{j}}u_{\hat{j}k}^* ) , \;\;\;   j,k=1, \ldots, d.
\end{equation}
Note then that for $j,k$ as above
\begin{eqnarray*}
\psi_u ((uF)_{jk}) &=& F_{\check{k}k}\psi_u (u_{j\check{k}})
= \left \{ \begin{array}{lll} 
-\frac{F_{\check{k}k}}{2}\|\eta_u(u_{\check{k}\check{k}})\|^2 & \mbox{ if } 
j=\check{k} \\
0 & \mbox{ if } j\neq \check{k} 
\end{array} \right. , \\
\psi_u ((F\bar{u})_{jk}) &=& F_{j\hat{j}}\overline{\psi_u (u_{\hat{j} k})}
= \left \{ \begin{array}{lll} 
-\frac{F_{j\hat{j}}}{2}\|\eta_u(u_{\hat{j}\hat{j}})\|^2 & \mbox{ if } \hat{j}=k 
\\
0 & \mbox{ if } j_0\neq k
\end{array} \right. .
\end{eqnarray*}
Now observe that $j=\check{k}$ if and only if $\hat{j}=k$. Thus the formulas 
above and the equality  \eqref{eq_eta_gaussian_aof} guarantee that 
\eqref{eq_psi_uF-Fu} always holds.
\end{proof}

Note that Theorem \ref{thm_aoq_decomp} provides an alternative proof of 
the property (GC) for $SU_q(2)$, $q\in (0,1)$, originally established by 
Sch\"{u}rmann 
and Skeide in \cite{schurmann+skeide98}. Indeed,  $SU_q(2)\cong O_F^+$ with the generic 
$F=\begin{pmatrix} 0 & 1 \\ -q^{-1} & 0 \end{pmatrix}$, see \cite[Proposition 6.4.8]{timmermann}, or with $F=\begin{pmatrix} 0 & 1 \\ -q^{-1} & 0 \end{pmatrix}$ (see Example \ref{ex_suqd}). 

We will now show that the case $SU_q(d)$ with $d\geq 3$ is different from that of $d=2$.

\begin{prop}\label{prop_suqn_gc}
Let $d\geq 3$ and $q\in (0,1)$. The quantum group $SU_q(d)$ does not have (GC). 
\end{prop}

\begin{proof}
We know from Example \ref{ex_suqd} that $\Pol(SU_q(d))$ is a  quotient of $\Pol(U_F^+)$ by the ideal generated by the twisted determinant condition \eqref{eq_td} with a generic matrix $F= (q^{j-d}\delta_{jk})_{j,k=1}^d$. 

Consider first the case of $d=3$ and let $u=(u_{jk})_{j,k=1}^3$. Set 
\[ \eta(u_{11}):=-1-i, \; \eta(u_{22}):= 1, \; \eta(u_{33}):=i, \; \eta(u_{jk})=0 
\;\;\textup{for } j \neq k  \]
(with $i$ being the complex unit) and $\eta(u_{jk}^*):= - \eta(u_{kj})$ for all $j,k=1\ldots,3$, and extend it to the whole algebra by the formula $\eta(ab)=\Cou(a)\eta(b)+\eta(a)\Cou(b)$ ($a, b \in \Pol(SU_q(3))$). We can check that such $\eta$ vanishes on the relations (R1) and \eqref{eq_td} defining $SU_q(3)$, hence by Lemma \ref{lem_5.8} it defines a Gaussian cocycle on $\Pol(SU_q(3))$ with values in $\bc$. We will write simply $\eta_j=\eta(u_{jj})$ for $1\leq j \leq 3$.

Suppose now that $\eta$ admits a generating functional $\psi:{\rm Pol} (SU_q(3))\to\bc$. 
Then, computing the value of $\psi$ on both sides of \eqref{eq_td}, and writing 
for simplicity $(t_1, t_2,t_3)$ for 
$(\tau(1), \tau(2), \tau(3))$ and  $(s_1, s_2,s_3)$ for 
$(\sigma(1), \sigma(2), \sigma(3))$,  for any $\tau\in S_3$ we would have
\begin{align*}
0 &= \psi\big(\sum_{\sigma\in S_3} (-q)^{i(\sigma)}
u_{s_1,t_1}u_{s_2,t_2} u_{s_3,t_3}\big)
\\ & = \sum_{\sigma\in S_3} (-q)^{i(\sigma)} 
\left[ \psi( u_{s_1,t_1}) \epsilon(u_{s_2,t_2} u_{s_3,t_3}) 
+ 
\epsilon(u_{s_1,t_1})\psi(u_{s_2,t_2} u_{s_3,t_3}) 
+ \big\langle \eta(u_{s_1,t_1}^*),\eta(u_{s_2,t_2} u_{s_3,t_3})\big\rangle 
\right]
%%%%%%%%%%%%%%%%%%%%%%%%
\\ & = 
(-q)^{i(\tau)}\psi( u_{t_1,t_1}) + 
\sum_{\sigma\in S_N} (-q)^{i(\sigma)} \delta_{s_1,t_1} 
\big[\psi(u_{s_2,t_2})\delta_{s_3,t_3})
+\delta_{s_2,t_2} \psi(u_{s_3,t_3})\big] 
\\ & \ \ \ + \sum_{\sigma\in S_N} (-q)^{i(\sigma)}  \delta_{s_1,t_1}  
\langle\eta(u_{s_2,t_2}^*),\eta( u_{s_3,t_3})\rangle 
+ 
\big\langle \eta(u_{s_1,t_1}^*),\epsilon(u_{s_2,t_2})\eta( 
u_{s_3,t_3})  +\eta(u_{s_2,t_2})\epsilon( u_{s_3,t_3})\big\rangle 
%%%%%%%%%%%%%%%%%%%%%%%%
\\ & = 
(-q)^{i(\tau)}\big[\psi( u_{t_1,t_1}) + \psi(u_{t_2,t_2})+\psi(u_{t_3,t_3}) - 
\langle \eta_{t_2},\eta_{t_3}\rangle - 
\langle \eta_{t_1},\eta_{t_3}\rangle -
\langle \eta_{t_1},\eta_{t_2}\rangle\big]. 
\end{align*}
We conclude that if $\eta$ admitted a generating functional, then the value
\begin{align*}C_\tau:&= \langle \eta_{t_1},\eta_{t_2}\rangle +
\langle \eta_{t_2},\eta_{t_3}\rangle + 
\langle \eta_{t_1},\eta_{t_3}\rangle  = \la - \eta_{t_2} - \eta_{t_3}, 
\eta_{t_2} \ra + \langle \eta_{t_2},\eta_{t_3}\rangle + \langle - \eta_{t_2} - 
\eta_{t_3},\eta_{t_3}\rangle 
\\&= - \|\eta_{t_2}\|^2 - \|\eta_{t_3}\|^2 -  \la \eta_{t_3}, \eta_{t_2} \ra  
\end{align*}
would be independent on $\tau\in S_3$.
But that is clearly not possible if we consider $\tau= \id$ and $\tau=(23)$, as 
$\la \eta_2, \eta_3 \ra = i \notin \br$.
This proves that $SU_q(3)$ does not have (GC). 
	
Embedding the above example into $SU_q(d)$ for $d>3$ shows that the latter does 
not have (GC) either.
\end{proof}

Note that, despite ongoing attempts to answer this question, it is still unknown whether $SU_q(d)$ (for $d \geq 3$) possesses properties (LK) or (NC).
%%%%%%%%%%%%%%%%%%%%%%%%%%%%%%%%%%%%%%

\section{Non-generic cases}
\label{sec-degenerate}

Throughout this section we will consider quantum groups $U_F^+$ and $O_F^+$ 
for non-generic matrices $F$. In fact we will mainly focus on the case of $F=I$, 
as it exhibits the typical non-generic behaviour and can be used to determine 
other cases.

Given a compact quantum group $\QG$ with $\QG=U_d^+$ or $\QG=O_d^+$ we will 
denote the canonical fundamental representation by $u=(u_{jk})_{j,k=1}^d \in 
M_d(\Pol(\QG))$. Further given a Hilbert space $H$, a representation 
$\rho:\Pol(\QG)\to B(H)$, and a $\rho$-$\Cou$-cocycle $\eta:\Pol(\QG)\to H$ we 
will consider the matrices
\begin{equation} \label{RVdef} R=(\rho(u_{jk}))_{j,k=1}^d \in M_d(B(H)),\;\;\;\; \;\;V=(\eta(u_{jk}))_{j,k=1}^d \in M_d(H), \end{equation} and
\begin{equation} \label{RWdef} W=(\eta(u_{jk}^*))_{j,k=1}^d \in M_d(H). \end{equation}
Finally, given a cocycle $\eta$ as above, it will also be useful to consider the scalar, selfadjoint, matrices
\begin{equation} \label{Bdef} B=\left(\sum_{p=1}^d \la \eta(u_{jp}), \eta(u_{kp}) \ra\right)_{j,k=1}^d \in M_d(\bc)\end{equation}
and 
\begin{equation} \label{B2def} \wt{B}=\left(\sum_{p=1}^d \la \eta(u_{jp}^*), \eta(u_{kp}^*). \ra\right)_{j,k=1}^d \in M_d(\bc).\end{equation}

Note that by Lemma \ref{lemma_eta_S} for Gaussian cocycles we have $W=-V^t$.

We will need later a simple lemma regarding an alternative description of the matrix $B$.

\begin{lem} \label{Lemma-relation}
Let $ d\in \bn$, let $H$ be a Hilbert space,  let $\rho:\Pol(U_d^+)\to B(H)$ be a representation, let $\eta:\Pol(\QG)\to H$ be a $\rho$-$\Cou$-cocycle and let $B$ be defined as in \eqref{Bdef}. Then we have for all $j,k=1,\ldots,d$
\[\eta(u^\ast_{jk})= - \sum_{p=1}^d \rho(u_{jp}^{\ast}) \eta(u_{kp} ),\]
\[ B_{jk} =  \sum_{p=1}^d  \la \eta(u_{pj}^\ast), \eta(u^{\ast}_{pk}) \ra.\]
\end{lem}
\begin{proof}
Using the defining relation $\bar{u}u^t=I$, we see that for all $j,k=1,\ldots,d$
\begin{equation*}
\begin{split}
0=\eta(\sum_{p=1}^d u^\ast_{jp}u_{kp})&=\sum_{p=1}^d\rho(u^\ast_{jp})\eta(u_{kp}
)+\eta(u^\ast_{jk}), 
\end{split}
\end{equation*}
and the first formula in the lemma is proved.

Then
\begin{align*}
\sum_{p=1}^d\la\eta(u^\ast_{pj}),\eta(u^\ast_{pk})\ra
&=\sum_{p=1}^d\la\sum_{m=1}^d\rho(u^\ast_{pm})\eta(u_{jm}),\sum_{m^\prime=1}^d
\rho(u^\ast_{pm^\prime})\eta(u_{km^\prime})\ra
\\&=\sum_{p,m,m^\prime=1}^d\la\rho(u_{pm^\prime}u^\ast_{pm})\eta(u_{jm}),\eta(u_{km^\prime})\ra
\stackrel{u^t\bar{u}=I}{=}\sum_{m=1}^d\la\eta(u_{jm}),\eta(u_{km})\ra
=B_{jk}.
\end{align*}
and the proof is complete.
\end{proof}

\begin{prop}\label{prop_v_unplus}
Let $H$ be a Hilbert space, let $\rho:\Pol(U_d^+)\to B(H)$ be a representation, define the matrix $R$ as in the first formula in \eqref{RVdef} and let $V\in M_d(H)$. Then the following conditions are equivalent:
\begin{itemize}
\item[(i)] there is a $\rho$-$\Cou$-cocycle $\eta$ on $\Pol(U_d^+)$ satisfying the second formula in \eqref{RVdef}; 
\item[(ii)] \begin{equation} \label{eq_relations_RV}
(R^*V)^t = \bar{R}V^t.
\end{equation}
\end{itemize}
Moreover when the above hold, we have $W= - \bar{R} V^t$, with $W$ as in \eqref{RWdef}.
\end{prop}

\begin{proof}

Suppose first that $\eta$ is a $\rho$-$\Cou$-cocycle on $\Pol(U_d^+)$ and define $R, V$ and $W$ via \eqref{RVdef} and \eqref{RWdef}. It follows from $u^*u=1$ 
that for each $j,k=1,\ldots, d$
\begin{align*}
0= & \eta (\delta_{jk} \one) = \eta\left( \sum_{p=1}^d u_{pj}^*u_{pk}\right)
= \sum_{p=1}^d \rho(u_{pj}^*)\eta (u_{pk}) + \eta ( u_{pj}^*) \Cou(u_{pk})
\\ =& \sum_{p=1}^d (R^*)_{jp} V_{pk} + W_{jk} = (R^* V + W^t)_{jk}
\end{align*}
and $\bar{u}u^t=1$ yields 
\begin{align*}
0= & \eta (\delta_{jk} \one) = \eta\left( \sum_{p=1}^d u_{jp}^*u_{kp}\right)
= \sum_{p=1}^d \rho(u_{jp}^*)\eta (u_{kp}) + \eta ( u_{jp}^*) \Cou(u_{kp})
\\ =& \sum_{p=1}^d \bar{R}_{jp} V_{kp} + W_{jk} = (\bar{R} V^t + W)_{jk}.
\end{align*}
Comparing the two formulea we see that $$ (R^* V)^t = -W= \bar{R} V^t,$$ i.e.\ \eqref{eq_relations_RV} holds.

On the other hand, if we are given a matrix $V\in M_d(H)$ which satisfies 
\eqref{eq_relations_RV}, then the mapping defined on the generators ($j,k=1,\ldots, d$) as $\eta(u_{jk})=V_{jk}$ and 
$\eta(u_{jk}^*)=W_{jk}$, with $W := -\bar{R}V^t$, extends by linearity and the 
cocycle property to a 
mapping (a $\rho$-$\Cou$-cocycle) on $\Pol(U_d^+)$. Indeed, as the 
calculations in the previous step show, $\eta$ defined this way respects the relations $u^*u-1$ 
and $\bar{u}u^t-1$. So we are left to prove that $\eta$ respects the two 
remaining sets of relations.

Note then that if the equality \eqref{eq_relations_RV} holds and $W := -\bar{R}V^t$, then 
$-V = (R^tW)^t = RW^t$. Indeed, as $R^t\bar{R}=RR^* = I$, we have
\[
(R^tW)^t =  -(R^t\bar{R}V^t)^t = -V,\]
\[
RW^t =  -R(\bar{R}V^t)^t \stackrel{\eqref{eq_relations_RV}}{=} 
-R((R^*V)^t)^t=-RR^* V = -V.
\]
Thus for any $j,k=1,\ldots,d$
\begin{align*}
\eta \big((uu^*)_{jk}\big) = &\sum_{p=1}^d  \eta ( u_{jp}u_{kp}^*)
= (RW^t + V)_{jk}=0 = \eta(\delta_{jk}\one) \quad \mbox{and}
\\
\eta((u^t\bar{u})_{jk}) = & \sum_{p=1}^d \eta ( u_{pj}u_{pk}^* )
= (R^t W + V^t)_{jk}=0 = \eta(\delta_{jk}\one),
\end{align*}
so $\eta$ respects also the missing relations.
\end{proof}

We are now ready to formulate a criterion for a given cocycle on $\Pol(U_d^+)$ to admit a generating functional.

\begin{tw}\label{Theorem:iff on UNplus}	
Let $H$ be a Hilbert space, let $\rho:\Pol(U_d^+)\to B(H)$ be a representation, and let $\eta:\Pol(U_d^+)\to H$ be a $\rho$-$\Cou$-cocycle.
Then $\eta$ admits a generating functional if and only if
\[\wt{B}=B^t,\] 
where $B$ and $\wt{B}$ are 
defined as in \eqref{Bdef} and \eqref{B2def}.
\end{tw}

\begin{proof}
Suppose that  $\psi$ is a generating functional associated with $\eta$. Then, due to the relations
$uu^*=I$ and $\bar{u}u^t=I$, we have for each $j,k=1, \ldots,d$
\begin{equation*}
\begin{split}
0&=\delta_{jk}\psi(\one)=\sum_{p=1}^d\psi(u_{jp}u^\ast_{kp})=\psi(u_{jk})+\overline{
\psi(u_ { kj } ) } +\wt{B}_{jk}, \\
0&=\delta_{kj}\psi(\one)= \sum_{p=1}^d\psi(u^\ast_{kp}u_{jp})=\overline{\psi(u_{kj})}
+\psi(u_{jk})+B_{kj}.
\end{split}
\end{equation*}
Comparing the above two equations, we arrive at $\wt{B}=B^t,$ as 
required.

Conversely suppose that the cocycle $\eta$ satisfies $\wt{B}=B^t$. Consider the algebra $\BGW$, with the generators denoted by $\{x_{jk}:j,k=1,\ldots,d\}$. Let $q:\BGW\longrightarrow \Pol(U_d^+)$ 
be the quotient map and set $\rho'=\rho\circ q$,  $\eta^\prime:=\eta\circ q$. Then $\rho'$ is a representation of $\BGW$ and $\eta'$ is a $\rho'$-$\Cou$-cocycle on $\BGW$. It follows from Theorem \ref{thm_schurmann} that the map 
defined by $(j,k=1,\ldots, d)$
\[
\psi^\prime(x_{jk})=-\frac{1}{2}\wt{B}_{jk}
\]
extends to a generating functional on $\BGW$ corresponding to the cocycle 
$\eta^\prime$. We aim to define a functional $\psi:{\rm Pol} 
(U_d^+)\rightarrow\bc$ by 
$\psi(q(x)):=\psi^\prime(x)$ for $x\in \BGW$. According to Lemma \ref{lem_5.8} $\psi$ will be 
well-defined provided we have for all  $j,k=1,\ldots, d$,
\[
\psi^{\prime}(\sum_{p=1}^d x^\ast_{jp}x_{kp})=0=\psi'(\sum_{p=1}^d x_{pj}x^\ast_{pk}). 
\]
Since $\overline{\tilde{B}_{jk}}=\tilde{B}_{kj}$, and by the assumption $\tilde{B}_{kj}=B_{kj}$ we have
\begin{equation*}
\begin{split}
\sum_{p=1}^d\psi^\prime(x^\ast_{jp}x_{kp})=\overline{
\psi^\prime(x_{jk})}+\psi^\prime(x_{kj})+B_{jk}=-\wt{B}_{kj}+B_{jk}=0,
\end{split}
\end{equation*}
and similarly
\begin{equation*}
\sum_{p=1}^d\psi^\prime(x_{pj}x^\ast_{pk})=\psi^\prime(x_{kj})+\overline{
\psi^\prime(x_{jk})}+\sum_{p=1}^d  \la \eta(x_{pj}^*), \eta(x^{\ast}_{pk}) \ra= 
-\wt{B}_{kj}+ \sum_{p=1}^d  \la \eta(x_{pj}^*), \eta(x^{\ast}_{pk}) \ra.
\end{equation*}
  It remains to use Lemma \ref{Lemma-relation} and the assumption that  $\wt{B}=B^t$ to conclude that the above expression equals zero.

Thus $\psi$ is well-defined and can be easily checked to be a generating functional associated to $\eta$.
\end{proof}

Finally note that Theorem \ref{Theorem:iff on UNplus} takes a very simple form for  cocycles with values in $\bc$.

\begin{cor}\label{corollary:Gaussian}
Let $\rho:\Pol(U_d^+)\to B(\bc)$ be a representation, and let $\eta:\Pol(U_d^+)\to \bc$ be a $\rho$-$\Cou$-cocycle. Then $\eta$ admits a generating 
functional if and only if $\bar{W}W^t=VV^*$, where $V$ and $W$ are defined as in \eqref{RVdef} and \eqref{RWdef}.
In particular a Gaussian cocycle $\eta:\Pol(U_d^+)\to \bc$ admits a generating 
functional if and only if $VV^*=V^*V$.
\end{cor}
\begin{proof}
It suffices to observe that in the case $H=\bc$ we have $B=\bar{V} V^t$, and $\wt{B}=\bar{W} W^t$ and apply Theorem \ref{Theorem:iff on UNplus}. In the Gaussian case we use further the equality $W=-V^t$.
\end{proof}

%%%%%%%%%%%%%%%%%%%%%%%%%%%%%%%%%%%%%%%%%%%%%%%%%%%%
\subsection{$U_F^+$ for non-generic $F$}

In this subsection we will show that for non-generic $F$ the quantum group $U_F^+$ does not have any of the properties (GC), (NC) or (LK). Of course to that end it suffices to show that $U_F^+$ does not have  Property (LK), but our proof will proceed by first building Gaussian and purely non-Gaussian cocycles without generating functionals.

\begin{tw} \label{thm-non-generic U_F}
Let $d\in \bn$ and let $F\in GL_d(\bc)$ be a non-generic matrix (so in particular $d\geq2$). Then the compact quantum group $U_F^+$ has none of the properties (GC), (NC) and (LK).
\end{tw}

\begin{proof}
We begin by considering the case of $F=I_2$ (so $\QG=U_2^+$), the general case will follow later.
In this proof we will denote the canonical generators of $\Pol(U_2^+)$ by $(v_{jk})_{j,k=1}^2$.  

\medskip

\textbf{No property (GC)}: 
Proposition \ref{prop_v_unplus} implies that any matrix $V\in M_2(\bc)$ defines a  
Gaussian cocycle $\eta_G$ on $U_2^+$ by the formula $\eta_G(v_{jk})= V_{jk}$, $j,k=1, \ldots, 2$. Corollary \ref{corollary:Gaussian} implies that if $V$ is not normal, then $\eta_G$ does not admit a generating functional. Fix the choice of $V=\begin{pmatrix}
1 & i \\ 0 & 1
\end{pmatrix}$.

\medskip 
\textbf{No property (NC)}: consider the purely non-Gaussian representation $\gamma$ of $\Pol(U_2^+)$ on $\bc$, given by the matrix $R=-I_2$ and the formula
\[\gamma(v_{jk}) = R_{jk}, \;\;\;\;\;\; j,k=1, \ldots, 2.\]
Note that here again by Proposition \ref{prop_v_unplus} any matrix $V'\in M_2(\bc)$ defines a purely non-Gaussian cocycle $\eta_N$ on $\Pol(U_2^+)$ by the formula $\eta_N(v_{jk})= 
V'_{jk}$, $j,k=1, \ldots, 2$. This time however, by the second part of 
Proposition \ref{prop_v_unplus}, we have $W'=(V')^t$. This change does not 
affect the necessary condition in Proposition \ref{prop_v_unplus} and again any 
choice of a non-normal $V'$ yields a cocycle $\eta_N$ which does not admit a 
generating functional. Fix the choice of $V'=\begin{pmatrix}
1 & 0 \\ i & 1
\end{pmatrix}$.

\medskip 
\textbf{No property (LK)}: consider $H=\bc\oplus \bc$, $\rho=\Cou\oplus\gamma$ and 
$\eta=\eta_G\oplus \eta_N$ with $\eta_G, \eta_N$ not admitting generating functionals constructed above. Then the matrix $R$ is of the form 
$R= \begin{pmatrix} J & 0 \\ 0 & J  \end{pmatrix}$ with $J=\begin{pmatrix}  1 & 
0 \\ 0 & -1 \end{pmatrix}$, and $V = Ve_1 + V'e_2.$  A direct computation using formulae \eqref{Bdef} and \eqref{B2def} shows that we have 
$$B=\begin{pmatrix} 3 & 0 \\ 0 & 3 
\end{pmatrix}   = \wt{B},$$ 
hence, by Theorem \ref{Theorem:iff on UNplus}, $\eta$ admits a generating functional. But this functional has no 
L\'evy-Khintchine decomposition. If it was the case, then $(\Cou,\eta_G)$ and 
$(\delta, \eta_N)$ would admit generating functionals, which was shown not to be 
the case.

\medskip

\textbf{General case:}  We can assume that $F^*F$ is of the form $\textup{diag}[\lambda_1,\ldots,\lambda_{d}]$, with $\lambda_1=\lambda_2=1$. The key role will be played by the unital $^*$-homomorphism $\pi:\Pol(U_F^+)\to \Pol(U_2^+)$ given by the formula 
\[\pi(u_{jk}) = \begin{cases} v_{jk} & \textup{ for }  j,k\in\{1,2\}, \\
\delta_{jk} 1 & \;\;\;\;\; \textup{else}. \end{cases}
\]	

\medskip
\textbf{No property (GC)}: 
Define a Gaussian cocycle $\wt{\eta}_G$ on $\Pol(U_F^+)$ by the formula $\wt{\eta}_G:= \eta_G \circ \pi$, with $\eta_G$ as in the first part of the proof. Suppose that $\wt{\eta}_G$ admits a generating functional $\psi_G: \Pol(U_F^+)\to \bc$. Recall that the generators of $U_F^+$ satisfy the two following conditions ($j,k=1,\ldots,d$):
\[ \sum_{p=1}^d u_{jp}u_{kp}^* = \delta_{jk} 1,\;\;\;  \sum_{p=1}^d 
\frac{\lambda_p}{\lambda_j} u_{pk}u_{pj}^* =\delta_{jk} 1.\]
We are only interested in fact in $j,k=1,2$ and apply $\psi_G$ to both of the above equalities. 
It is easy to see that as $\lambda_1=\lambda_2=1$ and $\wt{\eta}_G (u_{lm})$ vanishes unless $l,m\in \{1,2\}$, this leads to the following equalities  ($j,k=1,\ldots,2$):
\[ 0 = \psi_G(u_{jk}) + \psi_G(u_{kj}^*) + \sum_{p=1}^2 \la \wt{\eta}_G 
(u_{jp}^*), \wt{\eta}_G(u_{kp}^*) \ra
\]
and 
\[ 0 = \psi_G(u_{jk}) + \psi_G(u_{kj}^*) + \sum_{p=1}^2 \la \wt{\eta}_G 
(u_{pk}^*), \wt{\eta}_G(u_{pj}^*) \ra.
\]
Thus in particular we get 
\[ \sum_{p=1}^2 \la \eta_G (v_{jp}^*), \eta_G(v_{kp}^*) \ra = \sum_{p=1}^2 \la 
\eta_G (v_{pk}^*), \eta_G(v_{pj}^*) \ra.\]
This however is nothing but the condition for the existence of the generating functional for $\eta_G$, which we know not to hold. Thus $\wt{\eta}_G$ does not admit a generating functional.

\medskip
\textbf{No property (NC)}: we argue as above, defining first a purely non-Gaussian representation $\wt{\gamma}$ of $\Pol(U_F^+)$ on $\bc$ via the formula $\wt{\gamma}=\gamma \circ \pi$ and then a $\wt{\gamma}$-$\Cou$-cocycle $\wt{\eta}_N:=\eta_N\circ \pi$, with $\eta_N$ as in the first part of the proof. If we suppose that  $\wt{\eta}_N$ admits a generating functional $\psi_N: \Pol(U_F^+)\to \bc$, then once again we obtain the equality 
\[ \sum_{p=1}^2 \la \eta_N (v_{jp}^*), \eta_N(v_{kp}^*) \ra = \sum_{p=1}^2 \la 
\eta_N (v_{pk}^*), \eta_N(v_{pj}^*) \ra,\]
which contradicts the fact that $\eta_N$ does not admit a generating functional.

\medskip
\textbf{No property (LK)}: this is now straightforward. Define a representation $\wt{\rho}:\Pol(U_F^+) \to B(\bc\oplus \bc)$ as $\wt{\rho}:=\Cou \oplus \wt{\gamma}$ and the $\wt{\rho}$-$\Cou$-cocycle
$\wt{\eta}$ as $\wt{\eta}_G + \wt{\eta}_N$. By the first part of the proof $\wt{\eta}$ possesses a generating functional $\wt{\psi} = \psi \circ \pi$, where $\psi:\Pol(U_2^+)$ is the generating functional for $\eta_G+ \eta_N$.  Arguing as in the case of $U_2^+$ we see that $\wt{\psi}$ does not 
admit a L\'evy-Khintchine decomposition.
\end{proof}

\subsection{The free orthogonal quantum group $O_d^+$}

We begin from the case $d=2$.

\begin{prop}
The compact quantum group $O_2^+$ has  properties (GC) and (LK), but not (NC).
\end{prop}

\begin{proof}
Recall that $O_2^+\cong SU_{-1}(2)$. The statements follow from the results from \cite{schurmann+skeide98}  (see also \cite{skeide99}). (They can be also easily deduced from Proposition \ref{prop_v_onplus} and Corollary \ref{corollary:GaussianOrthogonal} below.)
\end{proof}

We shall see that the situation changes for $O_d^+$ with $d\geq 3$.
For that we first prove the analogue of Proposition \ref{prop_v_unplus} in the case of the free orthogonal group.  

\begin{prop} \label{prop_v_onplus}
Let $H$ be a Hilbert space, and let $\rho:\Pol(O_d^+)\to B(H)$ be a representation, define the matrix $R$ as in the first formula in \eqref{RVdef} and let $V\in M_d(H)$. Then the following conditions are equivalent:
\begin{itemize}
\item[(i)] there is a $\rho$-$\Cou$-cocycle $\eta$ on $\Pol(O_d^+)$ satisfying the second formula in\eqref{RVdef}; 
\item[(ii)] \begin{equation} \label{eq_relations_RV_orth}
(R^*V)^t = \bar{R}V^t=-V.
\end{equation}
\end{itemize}
\end{prop}

\begin{proof}
Any representation of $\Pol(O_d^+)$ lifts canonically to a representation of $\Pol(U_d^+)$, simply by composition with the quotient maps, and the same is true for the corresponding cocycles. Thus the condition $(R^*V)^t = \bar{R}V^t$ from Proposition \ref{prop_v_unplus} is necessary. Moreover every cocycle on $\Pol(O_d^+)$ must of course satisfy the conditions   
$\eta(u_{jk})=\eta(u_{jk}^*)$, for $j,k=1,\ldots, d$, which in terms of the proof of Proposition \ref{prop_v_unplus} means that $W=V$. Hence we must have 
$-V=-W=\bar{R}V^t=(R^*V)^t$. 
	
Conversely, if $V$ satisfies \eqref{eq_relations_RV_orth}, then -- as in the 
proof of Proposition \ref{prop_v_unplus} -- $\eta$ defined by $V$ vanishes on the relations that define 
$O_d^+$ and hence it can be extended by linearity and the cocycle property to 
the whole algebra in question.  
\end{proof}

Now, we check when a given $\rho$-$\Cou$-cocycle on $\Pol(O_d^+)$ admits a generating functional. We use the same methods as in Theorem \ref{Theorem:iff on UNplus}, with the only difference being that now matrices $B$ and $\wt{B}$ of \eqref{Bdef} and \eqref{B2def} naturally coincide (and are self-adjoint).

\begin{tw}\label{Theorem: iff on ONplus}	
Let $H$ be a Hilbert space, let $\rho:\Pol(O_d^+)\to B(H)$ be a representation, and let $\eta:\Pol(O_d^+)\to H$ be a $\rho$-$\Cou$-cocycle.
Then $\eta$ admits a generating functional if and only if the matrix $B$ defined in \eqref{Bdef} is symmetric, which is the case if and only if $B$ has real entries.
\end{tw}

The special case of Theorem \ref{Theorem: iff on ONplus}  for $\bc$-valued cocycles, an analogue of Corrolary \ref{corollary:Gaussian}, for the free orthogonal group takes the following form.

\begin{cor}\label{corollary:GaussianOrthogonal}
Let $\rho:\Pol(O_d^+)\to B(\bc)$ be a representation, and let $\eta:\Pol(O_d^+)\to \bc$ be a $\rho$-$\Cou$-cocycle. Then $\eta$ admits a generating 
functional if and only if $VV^*$ has real entries, where $V$ is defined as in \eqref{RVdef}.
\end{cor}

We are ready to show the main result of this section. 
\begin{tw}\label{thm-non-generic O_F}
Let $d\in \bn, d\geq3$.  Then the compact quantum group $O_d^+$ has none of the properties (GC), (NC) and (LK).
\end{tw}
\begin{proof}
\textbf{No property (GC)}: 
consider the block matrix 
\[V= \begin{pmatrix}
V_1 & 0 \\ 0 & 0
\end{pmatrix},\]
with 
\[	
V_1=	\begin{pmatrix}0&1& 1 \\-1 &0 & i & \\ -1 & -i & 0 \end{pmatrix}.\]
Proposition \ref{prop_v_onplus} implies, as $V=-V^t$, that the formula \eqref{RVdef} defines a $\bc$-valued Gaussian cocycle on $\Pol(O_d^+)$. On the other hand, as 
\[V_1 V_1^* = \begin{pmatrix}
2 & -i & i \\ i & 2 & 1 \\ -i &1 &2
\end{pmatrix},\]
Corollary \ref{corollary:GaussianOrthogonal} shows that this cocycle does not admit a generating functional.
\medskip
	
\textbf{No property (NC)}: 
consider the purely non-Gaussian representation $\gamma$ of $\Pol(O_d^+)$ on $\bc$, given by the matrix 
\[R=\begin{pmatrix} -1 & 0 & 0 & \cdots \\ 0 & 0 & 1 &\cdots \\ 0 & 1 & 0 &\cdots \\\vdots & \vdots & \vdots & \ddots 
\end{pmatrix}\]
and the formula
\[\gamma(u_{jk}) = R_{jk}, \;\;\;\;\;\; j,k=1, \ldots, d.\]
Further consider the block matrix
\[V'= \begin{pmatrix}
V_1' & 0 \\ 0 & 0
\end{pmatrix},\]
with 
\[	
V_1'=	\begin{pmatrix}i&1& -1 \\1 &0 & 0& \\ -1 & 0 & 0 \end{pmatrix}.\] 
Proposition \ref{prop_v_onplus} implies, as $V=-\bar{R}V^t = - (R^*V)^t$, that the formula \eqref{RVdef} defines a $\bc$-valued $\gamma$-$\Cou$-cocycle on $\Pol(O_d^+)$. On the other hand, as 
\[V_1' (V_1')^* = 
\begin{pmatrix}
3 & i & -i \\ -i & 1 & -1 \\ i &-1 &1
\end{pmatrix}
,\]
once again Corollary \ref{corollary:GaussianOrthogonal} shows that this cocycle does not admit a generating functional.

\medskip

\textbf{No property (LK)}: 
The proof now proceeds as in Theorem \ref{thm-non-generic U_F}; we consider the direct sum of the representations and cocycles in the last two cases, and observe that the corresponding matrix $B$ of \eqref{Bdef} for the resulting cocycle is the block-diagonal matrix of the form
\[B= \begin{pmatrix}
B_1 & 0 \\ 0 & 0
\end{pmatrix},\]
with 
\[	
B_1=V_1V_1^* + V_1'(V_1')^*	=\begin{pmatrix}5&0& 0 \\0 &3 & 0  \\ 0 & 0 & 3 \end{pmatrix}.\]
Thus by Theorem \ref{Theorem: iff on ONplus} the last cocycle admits a generating functional.
As before, the earlier parts of the proof imply that this generating functional cannot admit a L\'evy-Khintchine decomposition.
\end{proof}

The result above can be easily adopted, using techniques similar to these in Theorem \ref{thm-non-generic U_F}, to show that the quantum group $O_F^+$ with $F$ of the form 
\[\begin{pmatrix}
0 & D & 0 \\ D^{-1} & 0 & 0 \\ 0 & 0 & I_k
\end{pmatrix},\]
with $k\geq 3$, does not have any of the properties (GC), (NC) and (LK). The case of general $O_F^+$, with $F\in GL_d(\bc)$, $d\geq 3$, seems however beyond our reach for the moment. 

\subsection{Further remarks}

Methods developed in this section lead to certain further observations. 

Recall first that for any compact quantum group $\QG$ a cocycle $\eta:\Pol(\QG)\to H$ is said to be \emph{real} (\cite{Kyed}) if for any  $a,b\in\Pol(\QG)$ 
\begin{equation}\label{eq_real_cocycle}
\la\eta(a),\eta(b)\ra=\la\eta\big(S(b)^*\big),
\eta\big(S(a^*)\big)\ra\;\;\;
\end{equation}
where $S$ denotes the antipode of $\Pol(\QG)$. In \cite{DFKS} we showed that any real cocycle admits a generating functional. 
The following example shows that the reality of $\eta$ is not necessary; it 
also provides an example of noncommutative Gaussian process on $O_4^+$.
  
\begin{example}
Consider the following matrix $V\in M_4(\bc^2)$:
\[ V=\begin{pmatrix}
\binom{0}{0} & \binom{0}{0} & \binom{1}{0}& \binom{0}{1} \\[3pt]
\binom{0}{0} & \binom{0}{0} & \binom{i}{0}& \binom{0}{-i} \\[3pt]
\binom{-1}{0} & \binom{-i}{0} & \binom{0}{0}& \binom{0}{0} \\[3pt]
\binom{0}{-1} & \binom{0}{i} & \binom{0}{0}& \binom{0}{0} 
\end{pmatrix}.\]
Proposition \ref{prop_v_onplus} shows that the matrix $V$, as satisfying the condition $V=-V^t$, defines
a Gaussian cocycle $\eta:\Pol(O_4^+)\to \bc^2$ via the formula \eqref{RVdef}. A straightforward computation shows that then the corresponding matrix $B$, defined as in \eqref{Bdef} is equal to $2I_4$. Thus Theorem \ref{Theorem: iff on ONplus} shows that $\eta$ admits a generating functional $\psi:\Pol(O_4^+) \to \bc$. On the other hand $\eta$ is not real, as we have for example
$$\la\eta(u_{23}),\eta(u_{31})\ra = i \neq -i  =\la\eta(u_{13}),\eta(u_{32})\ra=
\la\eta(S(u_{31})^*),\eta(S(u_{23}^*))\ra. $$
We can also check that the functional $\psi$ constructed for $\eta$ as  in Theorem \ref{Theorem: iff on ONplus} is non-tracial: indeed, we have 
\[\psi(u_{23}u_{31}) = \la \eta(u_{23}), \eta(u_{31}) \ra \neq \la \eta(u_{31}), \eta(u_{23}) \ra = \psi(u_{31}u_{23}).\] 
 It was shown in \cite[Proposition 5.7]{FKS} that a Gaussian L\'evy process is commutative if and only its generating functional is a trace. So this shows that the Gaussian L\'evy process associated to this cocycle is indeed not commutative, as we claimed above.
\end{example}

Using the techniques of this section one can show that the \emph{half-liberated orthogonal quantum group} $O_d^*$ (see \cite{halfliberated}) for $d\geq 3$ does not have  properties (GC) or  (NC).

Finally let $\QG$ be a compact quantum group with a compact quantum subgroup $\QH$ (i.e.\ we have a surjective Hopf $^*$-algebra homomorphism $\pi:\Pol(\QG) \to \Pol(\QH)$). It is easy to see that any generating functional (any  cocycle, or any Sch\"{u}rmann triple) on $\QH$ can be transported in an obvious way to $\QG$ (the fact we used several times above). However  the relationship between properties 
(GC), (NC) or (LK) for the pair $(\QG, \QH)$ are far from straightforward. Indeed, if we consider the sequence of quantum subgroups $SU_q(2) \subset SU_q(3) \subset U_F^+$ (with $F$ the $3\times 3$ matrix from Example \ref{ex_suqd}), we see from the results of the last two sections that $SU_q(2)$ has (GC), $SU_q(3)$ does not have (GC), and $U_3^+(F)$ has (GC) again.
So far we do not know any case where $\QG$ would have property (LK) and its quantum subgroup $\QH$ would fail to have this property.

\section{Second cohomology group of $U_d^+$ and $O_d^+$ with trivial coefficients}\label{sec_cohom}

It is well-known that the L\'evy-Khintchine decomposition problem for a given unitary quotient algebra $A$ is closely related to the computation of the second Hochschild cohomology 
group of $A$ with trivial coefficients (see for example \cite{FGT15}). In this section we show certain consequences of our earlier results for the corresponding second Hochschild cohomology 
groups for $\Pol(O_d^+)$ and $\Pol(U_d^+)$. 

Let us briefly recall the definitions (we refer to \cite{FGT15} and to the survey \cite{Bichonsurvey} for 
more details). Given a unital $*$-algebra $A$ with a character $\Cou$ (e.g.\ a unitary quotient algebra), any representation $\rho:A\to B(H)$ on a Hilbert space $H$ allows to view $H$ as an $A$-bimodule with the left and right actions $a.z.b=\rho(a)\Cou(b)z$ for $a,b\in A$ and $v\in H$. As a special case, we can consider $\bc$ as an $A$-bimodule with the `trivial' left and right actions, i.e.\ 
$a.z.b=\Cou(a)\Cou(b)z$ for $a,b\in A$ and $z\in \bc$. 

For each $n \in \bn$ the  associated \emph{coboundary} operator $\partial^{n-1}: L(A^{\ot (n-1)};H) \to L(A^{\ot n}; H)$ (where we write $A^{\ot 0}: = \bc$) is determined by the formula
\begin{align*}
(\partial^{n-1} \phi)(a_1\otimes \ldots \otimes a_n) &= 
\rho(a_1)\phi(a_2\otimes \ldots \otimes a_n) + \sum_{j=1}^{n-1} (-1)^j 
\phi(a_1\otimes \ldots \otimes (a_ja_{j+1})\otimes  \ldots \otimes a_n)
\\ & + (-1)^n \phi(a_1\otimes \ldots \otimes a_{n-1})\Cou(a_n)
\end{align*}
for $\phi \in L(A^{\ot (n-1)}; H)$, $a_1, \ldots, a_n \in A$. We naturally have $\partial^n\circ \partial^{n-1}=0$.
We define further the vector space  of \emph{$n$-cocycles} 
$Z^n(A,{}_\rho H_\Cou)=\{ \phi \in L(A^{\ot n}; H):  \partial^n \phi=0\}$,  
the vector space of \emph{$n$-coboundaries}  $B^n(A,{}_\rho H_\Cou)=\{ 
\partial^{n-1} \psi: \psi\in L(A^{\ot (n-1)},H)\}$, and the 
\emph{$n$th-Hochschild cohomology group}  
$$H^n(A,{}_\rho H_\Cou)=Z^n(A,{}_\rho H_\Cou)/B^n(A,{}_\rho H_\Cou).$$
The 
\emph{$n$th-Hochschild cohomology group with trivial coefficients} is the special case
$$H^n(A,{}_\Cou\bc_\Cou)=Z^n(A,{}_\Cou\bc_\Cou)/B^n(A,{}_\Cou\bc_\Cou).$$
The terrminology `trivial coefficients' refers to the fact that we could replace $\bc$ above by any $A$-bimodule; and further in the case $A= \bc[\Gamma]$ the $A$-bimodule $_\Cou\bc_\Cou$ can be viewed equivalently as a trivial $\Gamma$-bimodule.

Note that if $A$ is a unitary quotient algebra, then by construction $H^n(A,{}_\rho H_\Cou)$ is in fact a vector space over $\bc$. Moreover $H^0(A,{}_\Cou\bc_\Cou) \cong \bc$ and 
$H^1(A,{}_\Cou\bc_\Cou)$ coincides with the space of 
$\bc$-valued Gaussian 
cocycles defined in Section \ref{ssec_schurman}. 

 To simplify the notation for a compact matrix quantum group $\QG$ and $n \in 
\bn$ we will simply write $H^n(\QG)$ for $H^n (\Pol(\QG),{}_\Cou\bc_\Cou )$; and 
similarly $Z^n(\QG)$ and $B^n(\QG)$ for  $Z^n (\Pol(\QG),{}_\Cou\bc_\Cou )$,  
$B^n (\Pol(\QG),{}_\Cou\bc_\Cou )$.

Then it follows from 
Propositions \ref{prop_v_unplus} and \ref{prop_v_onplus}, respectively, that for each $d \in \bn$ we have  
$$H^1 (U_d^+)\cong M_d(\bc), \quad \mbox{and} \quad 
H^1(O_d^+)\cong \{V\in M_d(\bc): V^t=-V\}\cong \bc^{\frac{d(d-1)}{2}}$$ 
(see also \cite{CHT}). Similarly the proof of Theorem \ref{thm_au_gauss} and 
arguments similar to those in the proof of Proposition \ref{prop_suqn_gc}  show 
that for $d \in \bn$ and $F\in GL_d(\bc)$ generic we have $H^1(U^+_F)\cong 
\bc^d$ 
and for $q\in (0,1)$, $d \geq 2$ we have $H^1(SU_q(d))\cong \bc^{d-1}$.

If for a given compact matrix quantum group $\QG$ the space $H^2(\QG)$ is trivial, then $\QG$ has all the properties (GC), (NC) and (LK), see \cite{FGT15}. Thus the results of the last section imply that  $H^2(U_d^+)$ is non-trivial for $d \geq 2$ and  $H^2(O_d^+)$ is non-trivial for $d \geq 3$. 

We will now recall a few general facts which will be helpful in computing $H^2(\QG)$ in the rest of this section. First note that one can restrict attention to \emph{normalised} 2-cocycles, i.e.\ those $c \in Z^2(\QG)$ which satisfy the condition $c(1\ot 1) = 0$. Indeed, given a general cocycle $c\in Z^2(\QG)$ one can always pass to the normalised 2-cocycle $c':=c - c(1 \otimes 1) \partial \Cou $, which naturally yields the same class in $H^2(\QG)$. Furthermore it is easy to check that if a 2-cocycle $c$ is normalised, then in fact $c(1 \otimes a)=c(a \otimes 1)=0$ for all $a \in \Pol(\QG)$.

The following result is Lemma 5.4 of \cite{bfg}.

\begin{lem}\label{lemmabfg}
	Suppose that $(A, \Cou)$ is a unital $^*$-algebra with a character, $\psi:A \to \bc$ is a linear functional with $\psi(1)=0$ and $c \in Z^2(A, {}_\Cou\bc_\Cou )$ is a normalised 2-cocycle. Define the map $T: A \to \textup{End}(\bc\oplus \ker \Cou \oplus \bc)$ via the formula 
\begin{equation} \label{Tdefined} T(a) = \begin{pmatrix} 
	\Cou(a) & c(a \otimes -) & - \psi(a) \\
	0 & a \cdot - & a-\Cou(a) 1\\
	0 & 0 & \Cou(a)
	\end{pmatrix}, \;\;\;\; a \in A.	
\end{equation}
Then the map $T$ is a homomorphism if and only if $c = \partial \psi$.
\end{lem}

The next lemma follows from an elementary computation (see for example the proof of Proposition 3.1 in \cite{FGT15}).

\begin{lem} \label{lem_2-cocycle} Let $A$ be a unitary quotient algebra and let $\rho:A \to B(H)$ be a representation of $A$ on a Hilbert space $H$.
If $\eta_1,\eta_2: A \to H$ are $\rho$-$\Cou$-cocycles on $A$, then the formula $$K(\eta_1,\eta_2)(a\otimes b):=\langle \eta_1(a^*),\eta_2(b)\rangle, \;\;\;a, b \in A$$ determines a 
2-cocycle $K(\eta_1,\eta_2)$ in $Z_2(A,{}_\Cou\bc_\Cou )$. 
\end{lem}

\subsection{Computing $H^2(U_d^+)$}

Let $d\in \bn$. We will compute now the group $H^2(U_d^+)$. Begin with the following lemma, describing properties of arbitrary normalised 2-cocycles in $Z^2(U_d^+)$.

\begin{lem}\label{lem_cocycle_condition}
Let $c\in Z^2(U_d^+)$ be normalised. Then for all $j,k\in\{1,\ldots, d\}$ we have relations
\begin{align}
\sum_{p=1}^d c(u^*_{pj} \otimes u_{pk}) &= \sum_{p=1}^d c(u_{jp} \otimes 
u^*_{kp}), %\tag{Z1} 
\label{eq_z1}
	\\
\sum_{p=1}^d c(u^*_{jp} \otimes u_{kp}) & = \sum_{p=1}^d c(u_{pj} 
\otimes u^*_{pk}). %\tag{Z2} 
\label{eq_z2}
\end{align}
\end{lem}

\begin{proof}
Fix $j,k \in\{1,\ldots, d\}$ and  apply $\partial c$ to $\sum_{p,r=1}^d u_{jp} \otimes u^*_{rp} \otimes u_{rk}$ to obtain
\begin{align*}
0&=\sum_{p,r=1}^d \partial c(u_{jp}\otimes u_{rp}^* \otimes u_{rk}) =
	\\
&= \sum_{p,r=1}^d \left[ \Cou(u_{jp}) c( u_{rp}^* \otimes u_{rk}) -
c(u_{jp}u_{rp}^* \otimes u_{rk})+
c(u_{jp}\otimes u_{rp}^* u_{rk})-
c(u_{jp}\otimes u_{rp}^*) \Cou(u_{rk}) \right]
	\\
&= \sum_{r=1}^d c( u_{rj}^* \otimes u_{rk}) -
\sum_{r=1}^d  c(\sum_{p=1}^d u_{jp}u_{rp}^* \otimes u_{rk})+
\sum_{p=1}^d  c(u_{jp}\otimes \sum_{r=1}^d u_{rp}^* u_{rk})-
\sum_{p=1}^d  c(u_{jp}\otimes u_{kp}^*)
	\\
&\stackrel{(*)}{=} \sum_{r=1}^d c( u_{rj}^* \otimes u_{rk}) -
\sum_{r=1}^d  c(1 \otimes u_{rk})+
\sum_{p=1}^d  c(u_{jp}\otimes 1)-
\sum_{p=1}^d  c(u_{jp}\otimes u_{kp}^*)
	\\
&=\sum_{r=1^d} c( u_{rj}^* \otimes u_{rk}) -
\sum_{p=1^d}  c(u_{jp}\otimes u_{kp}^*),
\end{align*}
where in  $(*)$ we used the unitarity relations $u^*u= u u^* = I$. This shows \eqref{eq_z1}.

Similarly applying $\partial c$ to the sum	$\sum_{p,r=1}^d u_{pj} \otimes u^*_{pr} \otimes u_{jr}$ and using the fact that $\bar{u}u^t=u^t\bar{u}=I$ yields \eqref{eq_z2}.

\end{proof}

The next result allows us to characterise coboundaries in $Z^2(U^+_d)$.

\begin{lem} \label{lem_coboundary_condition}
	A 2-cocycle $c\in Z^2(U_d^+)$ is a coboundary if and only if for each $j,k \in \{1, \ldots,d\}$	
    \begin{equation} \label{eq_coboundary_condition}
\sum_{p=1}^d c(u_{pj}^*\otimes u_{pk}) = \sum_{p=1}^d c(u_{kp}^*\otimes u_{jp}).
	\end{equation}
\end{lem}

\begin{proof} 
Note first that we can assume we are dealing with normalised 2-cocycles; the passage $c \mapsto c - \lambda \partial \Cou$ for $\lambda \in \bc$ does not affect neither being a coboundary, nor satisfying the equalities in \eqref{eq_coboundary_condition}.

($\Rightarrow$) Assume that $c\in Z^2(U_d^+)$ is a normalised coboundary, so that there exists a functional $\psi:\Pol(U_d^+)\to \bc$  such that $\psi(1)=0$ and $c = \partial \psi$. Then for every $j,k\in \{1,\ldots, d\}$
	\begin{align*}
	\sum_{p=1}^d c(u_{pj}^*\otimes u_{pk}) & =
	\sum_{p=1}^d \partial \psi (u_{pj}^*\otimes u_{pk})  
	= \sum_{p=1}^d \left( \psi (u_{pj}^*)\Cou( u_{pk})-\psi (u_{pj}^* u_{pk})+\Cou(u_{pj}^*)\psi( u_{pk})\right)
	\\ & 
	= \psi (u_{kj}^*)-\psi (\sum_{p=1}^d u_{pj}^* u_{pk})+\psi( u_{jk}) 
	\stackrel{u^*u=I}{=} \psi (u_{kj}^*)+\psi( u_{jk}),\\
	\mbox{and similarly} \\
	\sum_{p=1}^d c(u_{kp}^*\otimes u_{jp}) & =
	\sum_{p=1}^d \partial \psi (u_{kp}^*\otimes u_{jp})  
	= \sum_{p=1}^d \left( \psi (u_{kp}^*)\Cou( u_{jp})-\psi (u_{kp}^* u_{jp})+\Cou(u_{kp}^*)\psi( u_{jp})\right)
	\\ & 
	= \psi (u_{kj}^*)-\psi (\sum_{p=1}^d u_{kp}^* u_{jp})+\psi( u_{jk}) 
	\stackrel{\bar{u}u^t=I}{=} \psi (u_{kj}^*)+\psi( u_{jk}).
	\end{align*}
This shows that \eqref{eq_coboundary_condition} holds.
	
	($\Leftarrow$) Let then $c \in Z^2(U^+_d)$ be normalised and satisfying \eqref{eq_coboundary_condition}. Set $\psi(1)=0$,
	\begin{equation} \label{eq_functional_definition}
	\psi(u_{jk})=\psi(u_{kj}^*) := \frac12\sum_{p=1}^d c(u_{pj}^*\otimes u_{pk}), \;\;\; j,k \in\{1,\ldots, d\}.
	\end{equation}
	We are going to prove that the map $T=T_{c, \psi}$ defined via the 
prescription \eqref{Tdefined} for $a \in \{1,u_{jk}, u_{jk}^*:j,k=1,\ldots d\}$ 
extends to a homomorphism $\tilde{T}:\Pol(U_d^+) \to \textup{End}(\bc\oplus \ker \Cou 
\oplus \bc)$. For that it suffices to check that elements 
$t_{jk}:=T_{c,\psi}(u_{jk})$, $j,k \in \{1,\ldots,d\}$ satisfy the relations 
that define the algebra $\Pol(U_d^+)$. Then we will be able to define the 
functional $\tilde{\psi}:\Pol(U_d^+) \to \bc$ via the formula $\tilde{\psi}(a):= - 
(\tilde{T}(a))_{13}$, $a \in \Pol(U_d^+)$ and finally conclude by Lemma \ref{lemmabfg}  
that $\partial \tilde{\psi} = c$, so that in particular $c$ is a coboundary.
	
	We first check that the matrix $t:=(t_{jk})_{j,k=1}^d$ satisfies the condition\ $t^*t = I$, i.e.\ we have for each $j,k\in \{1,\ldots,d\}$ the equality $\sum_{p=1}^d t_{pj}^*t_{pk}=\delta_{jk}I$. 
	This is indeed the case, as
	\begin{align*}
	&\sum_{p=1}^d t_{pj}^*t_{pk}=  \sum_{p=1}^{d}  \ T(u^*_{p j}) T(u_{pk})
	\\
	&= \sum_{p=1}^{d}  \begin{pmatrix} 
	\Cou(u^*_{p j}) & c(u^*_{p j} \otimes -) & -\psi(u^*_{p j}) \\
	0 & u^*_{p j} \cdot - & u^*_{p j}-\Cou(u^*_{p j})1 \\
	0 & 0 & \Cou(u^*_{p j})
	\end{pmatrix}
	\begin{pmatrix} 
	\Cou(u_{pk}) & c(u_{pk} \otimes -) & -\psi(u_{pk}) \\
	0 & u_{pk} \cdot - & u_{pk}-\Cou(u_{pk})1 \\
	0 & 0 & \Cou(u_{pk})
	\end{pmatrix}
	\\
	& = \begin{pmatrix} 
	\sum_{p=1}^{d}  \Cou(u^*_{p j})\Cou(u_{pk}) & (\star) & (\star \star) \\
	0 & \sum_{p=1}^{d}  u^*_{p j}u_{pk} \cdot - & \sum_{p=1}^{d}  \left( u^*_{p 
		j}[u_{pk}-\Cou(u_{pk})1] + [u^*_{p j}-\Cou(u^*_{p j})1]\Cou(u_{pk})\right)\\
	0 & 0 & \sum_{p=1}^{d}  \Cou(u^*_{p j})\Cou(u_{pk})
	\end{pmatrix}
	\end{align*}
	and further
\begin{align*}
\sum_{p=1}^d &t_{pj}^*t_{pk}	 = \begin{pmatrix} 
	\Cou(\sum_{p=1}^{d} u^*_{p j}u_{pk}) & (\star) & (\star \star) \\
	0 & \sum_{p=1}^{d}  u^*_{p j}u_{pk} \cdot - & \sum_{p=1}^{d}  \left( [u^*_{p 
		j}u_{pk}-\Cou(u^*_{p j})1]\Cou(u_{pk})\right)\\
	0 & 0 &  \Cou(\sum_{p=1}^{d} u^*_{p j}u_{pk})
	\end{pmatrix}
	\\
	& 
	= \begin{pmatrix} 
	\delta_{jk} & (\star) & (\star \star) \\
	0 & \delta_{jk} 1 \cdot - & 0\\
	0 & 0 & \delta_{jk}
	\end{pmatrix}.
	\end{align*}
It remains then to check that $ (\star) =0$ and that $ (\star\star) =0$. The first 
	fact holds due to the fact that $c$ is a normalised cocycle and that the arguments are taken from $\ker \Cou$, since
	$$ (\star) =\sum_{p=1}^{d} \left( \Cou(u^*_{p j})c(u_{pk} \otimes -)+ c(u^*_{p j} 
	\otimes  u_{pk} \cdot -)\right)
	= \sum_{p=1}^{d} \left( c(u^*_{p j}u_{pk}\otimes -)+ c(u^*_{p j} 
	\otimes  u_{pk} ) \Cou(-)\right) = 0.
	$$
	The second formula holds true because of \eqref{eq_functional_definition}:
	\begin{align*}
	(\star\star) & =  \sum_{p=1}^{d} \left(-\Cou(u^*_{p j})\psi(u_{pk})+ 
c(u^*_{p j} 
	\otimes [u_{pk}-\Cou(u_{pk})1]) - \psi(u^*_{p j})\Cou(u_{pk})\right) 
	\\ & = - \psi(u_{jk})+ \sum_{p=1}^{d} c(u^*_{p j} \otimes u_{pk})  - 
	\psi(u^*_{k j})
	= -2\cdot \frac12 \sum_{r=1}^d c( u_{rj}^* \otimes u_{rk}) + \sum_{p=1}^{d} c(u^*_{p 
		j} \otimes u_{pk})=0.
	\end{align*}
The fact that $tt^* = I$ follows in the same way -- note that we have not yet used the equality \eqref{eq_coboundary_condition}.

Next, we verify that $t^t\bar{t}=I$, i.e.\ for each $j,k \in\{1,\ldots,d\}$ we have $\sum_{p=1}^d t_{pj}t_{pk}^*=\delta_{jk}I$. Indeed, we have
	\begin{align*}
	\sum_{p=1}^d &t_{pj}t_{pk}^*=  \sum_{p=1}^{d}  \ T(u_{p j}) T(u_{pk}^*)
	\\
	&= \sum_{p=1}^{d}  \begin{pmatrix} 
	\Cou(u_{p j}) & c(u_{p j} \otimes -) & -\psi(u_{p j}) \\
	0 & u_{p j} \cdot - & u_{p j}-\Cou(u_{p j}) 1\\
	0 & 0 & \Cou(u_{p j})
	\end{pmatrix}
	\begin{pmatrix} 
	\Cou(u_{pk}^*) & c(u_{pk}^* \otimes -) & -\psi(u_{pk}^*) \\
	0 & u_{pk}^* \cdot - & u_{pk}^*-\Cou(u_{pk}^*)1 \\
	0 & 0 & \Cou(u_{pk}^*)
	\end{pmatrix}
	\\
	%& = \begin{pmatrix} 
	%\sum_{p=1}^{d}  \Cou(u^*_{p j})\Cou(u_{pk}) & (\star) & (\star \star) \\
	%0 & \sum_{p=1}^{d}  u^*_{p j}u_{pk} \cdot - & \sum_{p=1}^{d}  \left( u^*_{p 
	%j}[u_{pk}-\Cou(u_{pk})] + [u^*_{p j}-\Cou(u^*_{p j})]\Cou(u_{pk})\right)\\
	%0 & 0 & \sum_{p=1}^{d}  \Cou(u^*_{p j})\Cou(u_{pk})
	%\Cound{pmatrix}
	%\\
	& = \begin{pmatrix} 
	\Cou(\sum_{p=1}^{d} u_{p j}u_{pk}^*) & (\diamond) & (\diamond\diamond) \\
	0 & \sum_{p=1}^{d}  u_{p j}u_{pk}^* \cdot - & \sum_{p=1}^{d}  \left( u_{p 
		j}u_{pk}^*-\Cou(u_{p j})\Cou(u_{pk}^*)1\right)\\
	0 & 0 &  \Cou(\sum_{p=1}^{d} u_{p j}u_{pk}^*)
	\end{pmatrix}
	\\
	& 
	= \begin{pmatrix} 
	\delta_{jk} & (\diamond) & (\diamond\diamond) \\
	0 & \delta_{jk} 1 \cdot - & 0\\
	0 & 0 & \delta_{jk}
	\end{pmatrix}
	\end{align*}
	and, as above, 
	$$ (\diamond) =\sum_{p=1}^{d} \left( \Cou(u_{p j})c(u_{pk}^* \otimes -)+ c(u_{p j} 
	\otimes  u_{pk}^* \cdot -)\right)
	= \sum_{p=1}^{d} \left( c(u_{p j}u_{pk}^*\otimes -)+ c(u_{p j} 
	\otimes  u_{pk}^* ) \Cou(-)\right) = 0.
	$$
	and 
	\begin{align*}
	(\diamond\diamond) & =  \sum_{p=1}^{d} \left(-\Cou(u_{p j})\psi(u_{pk}^*)+ c(u_{p j} 
	\otimes [u_{pk}^*-\Cou(u_{pk}^*)1]) - \psi(u_{p j})\Cou(u_{pk}^*)\right) 
	\\ & = -\psi(u_{jk}^*)+ \sum_{p=1}^{d} c(u_{p j} \otimes u_{pk}^*) - 
	\psi(u_{k j})
	\stackrel{\eqref{eq_functional_definition}}{=} -2\cdot \frac12 \sum_{p=1}^d c( u_{pk}^* \otimes u_{pj}) + \sum_{p=1}^{d} c(u_{p 
		j} \otimes u_{pk}^*). 
	\end{align*}
	But, due to \eqref{eq_coboundary_condition}
	%$$\sum_{p=1}^d c(u_{pj}^*\otimes u_{pk}) = \sum_{p=1}^d c(u_{kp}^*\otimes u_{jp})$$
	and \eqref{eq_z2},
	%$$\sum_{p=1}^d c(u^*_{kp} \otimes u_{jp}) & = \sum_{p=1}^d c(u_{pk} \otimes u^*_{pj}) $$
	we see that 
	\begin{align*}
	(\diamond\diamond) & =  \sum_{p=1}^{d} c(u_{pk} \otimes u_{pj}^*) - \sum_{p=1}^{d} c( u_{pj}^* \otimes u_{pk})
	\stackrel{\eqref{eq_z2}}{=}  \sum_{p=1}^d c(u^*_{kp} \otimes u_{jp}) - \sum_{p=1}^{d} c( u_{pj}^* \otimes u_{pk}) \stackrel{\eqref{eq_coboundary_condition}}{=} 0.
	\end{align*}
Finally the equality $\bar{t} t^t = I$ can be verified in the same manner. %As explained earlier, this ends the proof.
	
\end{proof}

Before we pass to the main result of this subsection we note another 
interesting fact which can be deduced from the proof of the last two lemmas. 
Indeed, it follows from the proofs of Lemma \ref{lem_cocycle_condition} and 
Lemma \ref{lem_coboundary_condition} that any 2-cocycle on the 
Brown-Glockner-von Waldenfels algebra $\BGW$ (for any $d\in \bn$) is a 
coboundary. Indeed, any 2-cocycle on $\BGW$ satisfies the condition 
\eqref{eq_z1}, since the proof of the first part of Lemma 
\ref{lem_cocycle_condition} used only the unitarity of the matrix $u$, which 
holds for generators in $\BGW$. Furthermore, given a 2-cocycle $c$ we can follow 
the proof of Lemma \ref{lem_coboundary_condition}: define $\psi$ as in 
\eqref{eq_functional_definition} and then  show that $t_{jk}$'s satisfy the 
relations defining $\BGW$, so that one can deduce that $c$ is a coboundary (as 
noted above, this does not involve condition \eqref{eq_coboundary_condition}). 
Thus we obtain the following proposition.

\begin{prop} \label{prop_cohomology_Kd}
Let $d \in \bn$. The second cohomology group for the Brown-Glockner-von Waldenfels 
algebra $\BGW$ with trivial coefficients is trivial: 
$H^2(\BGW,{_\Cou}\bc_\Cou)=0$.
\end{prop}

We are almost ready for the proof of the main result of this subsection; we still need to introduce the 2-cocycles which will lead to a basis of $H^2(U_d^+)$. Assume that $d \geq 3$.

Denote by $e_{jk}$ ($j,k=1,\ldots,d$) the canonical matrix units in $M_d(\bc)$. 
%Denote by $e_{jk}\in M_d(\C)$ the matrix unit (i.e.\ 
% $(e_{jk})_{mn}=\delta_{jm}\delta_{kn}$). 
 By Proposition \ref{prop_v_unplus}, each $e_{jk}$ defines a  Gaussian 1-cocycle 
$\eta_{jk}$ on $\Pol(U_d^+)$. Thus further, by Lemma \ref{lem_2-cocycle}, if $m,n,l 
\in \{1,\ldots,d\}$ (and we view $l$ as an arbitrary index 
depending on $m$ and $n$), then  
\begin{equation}
  \label{Kmndef} K_{m,n} = K(\eta_{e_{lm}}, \eta_{e_{ln}})
\end{equation}
is a 2-cocycle.
 Moreover, for each $p\in \{1, \ldots,d-1\}$ define the matrix $V_p\in M_d(\bc)$ by 
inserting the matrix $V=\begin{pmatrix} 1 & i 
 \\ 0 & 1 \end{pmatrix}$ into the $(p,p+1)$-block. More precisely, we have 
$(V_p)_{p,p}=1$, $(V_p)_{p, p+1}=i$, $(V_p)_{p+1, p+1}=1$ and all the other 
entries are 0. Any such $V_p$ defines a Gaussian 1-cocycle $\eta_{V_p}$, and 
hence a 2-cocycle 
\begin{equation}
\label{Kpdef} K_p:=K(\eta_{V_p}, \eta_{V_p})
\end{equation}
on $\Pol(U_d^+)$.

\begin{tw}\label{thm_main_cohomology_Ud}
	Let $d \in \bn$. Then  $$H^2(U_d^+,{}_\Cou\bc_\Cou) \cong \bc^{d^2-1}.$$ Furthermore for $d \geq 2$ the set 
	$\wt{Y}_d=\{[K_{m,n}]:m\neq n, m,n=1,\ldots,d\} \cup \{[K_p]:p=1, \ldots, d-1\}$
	is a basis of $H^2(U_d^+,{}_\Cou\bc_\Cou)$. 
\end{tw}
\begin{proof}
Define $\Delta:Z^2(U_d^+) \to M_d(C)$ by the formula
$$\Delta (c)= \left(\sum_{p=1}^d \big( c(u_{pj}^*\otimes u_{pk}) - c(u_{kp}^*\otimes  u_{jp}) \big)\right)_{j,k=1}^d.$$
	
 Observe that $\Delta\big(Z^2(U_d^+)\big) \subseteq sl(d,\bc)$, where $sl(d,\bc)$ denotes the 
space of $d\times d$ complex matrices with trace zero; indeed, for any $c \in Z^2(U_d^+)$
%from Lemma \ref{lem_coboundary_condition}:
	$$\textup{Tr}(\Delta (c) ) = \sum_{j,p=1}^d c(u_{pj}^*\otimes u_{pj}) - 	\sum_{j,p=1}^d c(u_{jp}^*\otimes u_{jp})= 
	%\sum_{j=1}^d \left(\sum_{p=1}^d c(u_{pj}^*\otimes u_{pj}) - c(u_{jp}^*\otimes u_{jp}) \right) = 
	0.$$
	
 Next we check that  $\Delta\big(Z^2(U_d^+)\big) = sl(d,C)$, i.e.\ that $\Delta$ is surjective. 
Assume then that $d\geq 2$, recall \eqref{Kmndef}-\eqref{Kpdef} and compute 
($p\in\{1, \ldots,d-1\}$, $j,k,m,n\in \{1, \ldots,d\}$ and $m\neq n$)
	\begin{align*}
	(\Delta(K_p))_{jk} & 
	= \sum_{s=1}^d \big( \langle \eta_p(u_{sj}),\eta_p( u_{sk})\rangle - \langle \eta_p(u_{ks}),\eta_p( u_{js})\rangle \big)
	\\ &=(V_p^*V_p-(\bar{V}_pV_p^t)^t)_{jk} =[V_p^*,V_p]_{jk}= \delta_{jp}\delta_{kp} - \delta_{j,p+1}\delta_{k,p+1},
	\\
	(\Delta(K_{m,n}))_{jk} &= \sum_{s=1}^d \big( \langle \eta_{lm}(u_{sj}),\eta_{ln}( u_{sk})\rangle - \langle \eta_{lm}(u_{ks}),\eta_{ln}( u_{js})\rangle \big)
	\\ &= [e_{lm}^*e_{ln} - (\bar{e}_{lm}e_{ln}^t)^t]_{jk}=[e_{ml}e_{ln} - e_{ln}e_{ml}]_{jk}=(e_{mn})_{jk}.
	\end{align*}
This shows that
	$$\Delta(K_p) = e_{pp} - e_{p+1,p+1} , \quad \Delta(K_{m,n}) = e_{mn},$$
	and surjectivity of $\Delta$ follows from the fact that the matrices appearing on the right-hand-side form together a basis for
	$sl(d,C)$. 
	
	Finally, we note that Lemma \ref{lem_coboundary_condition} shows that 
$\ker \Delta = B^2(U_d^+)$. Hence $\Delta$ induces a linear isomorphism from 
$H^2(U_d^+)=Z^2(U_d^+)/B^2(U_d^+)$ to $sl(d,\bc)$ and the conclusions of the 
theorem hold for $d\geq 2$.

The case $d=1$ is trivial.
\end{proof}

\subsection{Computing $H^2(O_d^+)$}

The  method presented in the last subsection provides an alternative, elementary proof of the fact that 
$$ H^2(O_d^+) \cong \bc^{\frac{d(d-1)}{2}}, $$
as proved by Collins, H\"artel and Thom in \cite[Theorem 3.2]{CHT}  (see also \cite[Proposition 6.4]{BichonCompositio}).  It also provides a basis of the space $ H^2(O_d^+)$. 

Comparing to the unitary case, for $O_d^+$ the 'defect' map $\Delta$, which 
measures how far a cocycle is  from being a coboundary, maps the space of 
2-cocycles to $o(d)=o(d,\bc)$, the space of anti-symmetric complex $d\times d$ 
matrices. This difference comes from the fact that for $O_d^+$ we have the 
additional relation $u_{jk}^* = u_{jk}$. We sketch below the corresponding 
arguments.

\begin{lem}
	For every normalised $c\in Z^2(O_d^+)$ we have for all $j,k\in\{1,\ldots,d\}$
	\begin{equation} 
	\label{eq_cocycle_condition_O}
	\sum_{p=1}^d c(u_{pj} \otimes u_{pk}) = \sum_{p=1}^d c(u_{jp} \otimes u_{kp}).
	\end{equation}
\end{lem}

\begin{proof}
	Apply $\partial c$ to $\sum_{s,p=1}^d u_{js} \otimes u_{ps} \otimes u_{pk}$. 
\end{proof}

\begin{lem} \label{lem_coboundary_condition_O}
	A 2-cocycle $c\in Z^2(O_d^+)$ is a coboundary if and only if for all $j,k\in\{1,\ldots,d\}$
	\begin{equation} 
	\label{eq_coboundary_condition_O}
	\sum_{p=1}^d c(u_{jp} \otimes u_{kp}) = \sum_{p=1}^d c(u_{kp} \otimes u_{jp}).
	\end{equation}
\end{lem}

\begin{proof}Assume (as we may) that $c$ is normalised.

	($\Rightarrow$) Apply $-\partial \psi = c$ to $\sum_{p=1}^d u_{jp} \otimes u_{kp}$ and to $\sum_{p=1}^d u_{pj}
	\otimes u_{pk}$ and subtract the resulting equations.

	($\Leftarrow$) Set 
	$\psi(u_{jk}) =  \frac{1}{2} \sum_{p=1}^d c(u_{pj} \otimes u_{pk}) $
	(note that the two sums above coincide and define a symmetric matrix) 
	and use Lemma \ref{lemmabfg} to prove that $\partial \psi = c$. 
\end{proof}

Let for the moment $d \geq 3$. For any $j,k \in \{1, \ldots, d\}$ consider the matrices
$$Z_{jk}=e_{jk}-e_{kj}$$
and the associated Gaussian 1-cocycles (see Proposition \ref{prop_v_onplus}) $\eta_{Z_{jk}}$ on $O_d^+$. Then for any $m,n \in \{1, \ldots, d\}$, $m <n$, choose $l\in \{1,\dots,d\}$ different from both $m$ and $n$, and define
$$\widehat{K}_{mn}=K(\eta_{Z_{lm}},\eta_{Z_{ln}}).$$
Note that according to Lemma \ref{lem_2-cocycle}, each $\widehat{K}_{mn}$ is a 
2-cocycle. 

\begin{tw}
Let $d \in \bn$. Then  $$H^2(O_d^+,{}_\Cou\bc_\Cou) \cong \bc^{\frac{d(d-1)}{2}}.$$ Furthermore for $d \geq 3$ the set 
	$\wt{Y}_d=\{[\widehat{K}_{m,n}]:m< n, m,n=1,\ldots,d\}$ is a basis of 
$H^2(O_d^+,{}_\Cou\bc_\Cou)$. 
\end{tw}
\begin{proof}
 We just sketch the proof, analogous to that of Theorem 
\ref{thm_main_cohomology_Ud}, with the map $\Delta_O:Z^2(O_d^+)\to M_d$ given 
by 
$$\Delta_O(c) = \left(\sum_{p=1}^d \big( c(u_{jp} \otimes u_{kp}) - c(u_{kp} \otimes u_{jp}) \big)\right)_{j,k=1}^d,$$ 
which clearly yields an anti-symmetric matrix. If $d\geq 3$ then we check that 
\begin{align*}
\Delta_O (\widehat{K}_{mn}) & =\left(\sum_{p=1}^d \big( \langle 
\eta_{Z_{lm}}(u_{jp}),\eta_{Z_{ln}}(u_{kp}) \rangle 
- \langle \eta_{Z_{lm}}(u_{kp}),\eta_{Z_{ln}}(u_{jp}) \rangle\right)_{j,k=1}^d
\\ & = \bar{Z}_{lm} Z_{ln}^t - (\bar{Z}_{lm} Z_{ln}^t)^t
= (e_{lm}-e_{ml}) (e_{nl}-e_{ln}) - (e_{ln}-e_{nl})(e_{ml}-e_{lm})
\\ & = e_{mn} -e_{nm} = Z_{mn}. 
\end{align*}
Since the set of matrices $Z_{mn}$ defines a basis of $o(d)$, the map 
$\Delta_O$ is surjective.
Finally, by Lemma \ref{lem_coboundary_condition_O}, 
$B^2(O_d^+)=\Delta_O\big(Z^2(O_d^+)\big)$, and thus $\Delta_O$ establishes the 
homomorphism between $H^2(O_d^+)=Z^2(O_d^+)/ \ker \Delta_O$ and $o(d)$. 

Now if $d=2$ it suffices to exhibit a 2-cocycle $c \in Z^2(O_2^+)$ such that $\Delta_O(c) \neq 0$. To this end it suffices to consider matrices $Z_1=\begin{pmatrix}
1 & 0 \\ 0 & -1
\end{pmatrix}$ and $Z_2=\begin{pmatrix}
	0 & 1 \\ 1 & 0
\end{pmatrix}$,
corresponding \emph{anti-Gaussian} cocycles (see Proposition \ref{prop_v_onplus}) $\eta_{Z_1}$ and $\eta_{Z_2}$, and the 2-cocycle $K(\eta_{Z_1},\eta_{Z_2})$ defined via  Lemma \ref{lem_2-cocycle}. Then one can check that 
\[ \Delta_O(K(\eta_{Z_1},\eta_{Z_2}) ) = \begin{pmatrix}
0 & 2 \\ -2 & 0
\end{pmatrix} \]
and the proof is finished (the case $d=1$ is trivial).
\end{proof}

\section*{Acknowledgements}
BD would like to thank the support of Mathematics Research Unit, University of Oulu in the years 2016--2017, when a part of the
work was done. 
UF was supported by the French ``Investissements d'Avenir'' program, project 
ISITE-BFC (contract ANR-15-IDEX-03).  AK was supported by the National Science Centre grant SONATA 
2016/21/D/ST1/03010. AS was partially supported by the NCN (National Science Centre) grant 2014/14/E/ST1/00525. UF and AS acknowledge support by the French 
MAEDI and MENESR and by the Polish MNiSW through the Polonium programme. We 
thank Julien Bichon and Malte Gerhold for useful comments.

\end{document}